\documentclass[13pt, a4paper]{article}
 \usepackage{amsmath}
 \usepackage{amssymb}
\usepackage[margin=21mm]{geometry}
\usepackage[utf8]{inputenc}
\usepackage{graphics}
\usepackage{xspace}
\usepackage{color}
 \usepackage{amsthm}
 \usepackage[old]{old-arrows}
\usepackage[dvipsnames]{xcolor}
\usepackage{mathtools}
\usepackage{latexsym}
\newcommand{\R}{{\mathbb R}}

\newcommand{\N}{{\mathbb N}}

\newcommand{\inn}{\text{  in   }}
\newcommand{\dyle}{\displaystyle}
\newcommand{\dint}{\dyle\int}

\newtheorem{Theorem}{Theorem}[section]
\numberwithin{Theorem}{section}
\newtheorem{Corollary}[Theorem]{Corollary}
\newtheorem{Lemma}[Theorem]{Lemma}
\newtheorem{Proposition}[Theorem]{Proposition}
\newtheorem{Definition}[Theorem]{Definition}

\newtheorem{remark}[Theorem]{Remark}
\usepackage[mathscr]{euscript}
\usepackage{mathrsfs}
 \usepackage{enumerate}
 \usepackage{cite}
 \usepackage{tikz}
 \usepackage[pagebackref]{hyperref}
 \usepackage{hyperref}
\newtheorem{theorem}{Theorem}[section]
\numberwithin{equation}{subsection}

\newtheorem{definition}[theorem]{Definition}

\catcode`@=12

\makeatother
\usepackage[english]{babel}

 \usepackage[hyperpageref]{backref}
\title{On an eigenvalue problem associated with mixed operators under mixed  boundary conditions}
\date{}
\author{Jacques Giacomoni$^{2}\thanks{Corresponding author}$, Tuhina Mukherjee$^{1}$, Lovelesh Sharma$^{1}$  \\
       \small $^{1}$ Department of Mathematics, Indian Institute of Technology Jodhpur, Rajasthan 342030, India \\
       \small $^{2}$ LMAP (UMR E2S UPPA CNRS 5142) Bat. IPRA, Avenue de l’Université, F-64013 Pau, France\\
}
\newcommand{\Addresses}{{
  \bigskip
  \footnotesize
  J. ~Giacomoni, \textit{E-mail address:}
  \texttt{jacques.giacomoni@univ-pau.fr}\\
  \medskip
  T.~Mukherjee, \textit{E-mail address:} \texttt{tuhina@iitj.ac.in}\\
  L. Sharma, \textit{E-mail address:} \texttt{sharma.94@iitj.ac.in}
}}
\providecommand{\keywords}[1]
{
  \small	
  \textbf{\textit{Keywords---}} #1
}
\begin{document}
\maketitle \vspace{-1.8\baselineskip}
\begin{abstract}
In this paper, we study a class of eigenvalue problems involving both local as well as nonlocal operators, precisely the classical Laplace operator and the fractional Laplace operator in the presence of mixed  boundary conditions, that is 
 \begin{equation}  \label{1}
    \left\{\begin{split} \mathcal{L}u\: &= \lambda u,~~u>0~ \text{in} ~\Omega, \\
      u&=0~~\text{in} ~~{U^c},\\
 \mathcal{N}_s(u)&=0 ~~\text{in} ~~{\mathcal{N}}, \\
 \frac{\partial u}{\partial \nu}&=0 ~~\text{in}~~ \partial \Omega \cap \overline{\mathcal{N}},
    \end{split} \right.\tag{$P_\lambda$}
\end{equation}

 where  $U= (\Omega \cup {\mathcal{N}} \cup (\partial\Omega\cap\overline{\mathcal{N}}))$, $\Omega \subseteq \mathbb{R}^n$ is a non empty open set, $\mathcal{D}$, $\mathcal{N}$ are open subsets of $\mathbb{R}^n\setminus{\bar{\Omega }}$ such that $\overline{{\mathcal{D}} \cup {\mathcal{N}}}= \mathbb{R}^n\setminus{\Omega}$, $\mathcal{D} \cap {\mathcal{N}}=  \emptyset $ and $\Omega\cup \mathcal{N}$ is a bounded set with smooth boundary, $\lambda >0$ is a  real parameter and
 $$\mathcal{L}=  -\Delta+(-\Delta)^{s},~ \text{for}~s \in  (0, 1).$$ We establish the existence and some characteristics of the first eigenvalue and associated eigenfunctions to the above problem, based on the topology of the sets $\mathcal{D}$ and $\mathcal{N}$. Next, we apply these results to establish  bifurcation type results,  both
 from zero and infinity for the problem \eqref{ql} which is an asymptotically linear problem inclined with $(P_\lambda)$. 
\end{abstract}

\\

\section{Introduction}
We investigate the existence and main properties of the eigenvalues and eigenfunctions to the following problem
 \begin{equation*} 
\left\{\begin{split} \mathcal{L}u\: &= \lambda u,~~u>0~ \text{in} ~\Omega, \\
      u&=0~~\text{in} ~~{U^c},\\
 \mathcal{N}_s(u)&=0 ~~\text{in} ~~{\mathcal{N}}, \\
 \frac{\partial u}{\partial \nu}&=0 ~~\text{in}~~ \partial \Omega \cap \overline{\mathcal{N}},
    \end{split} \right.\tag{$P_\lambda$}
         \end{equation*}
  where  $U= (\Omega \cup {\mathcal{N}} \cup (\partial\Omega\cap\overline{\mathcal{N}}))$, $\Omega \subseteq \mathbb{R}^n$, $\mathcal{D}$, $\mathcal{N}$, respectively denoted open Dirichlet and Neumann set, are disjoint open subsets such that $\overline{{\mathcal{D}} \cup {\mathcal{N}}}= \mathbb{R}^n\setminus{\Omega}$
  	and $\Omega\cup \mathcal{N}$ is a bounded set with smooth boundary, $\lambda >0$ is a  real parameter, $\nu$ denotes the outward normal on $\partial\Omega\cap \overline{\mathcal{N}}$ and
 \begin{equation}\label{A}
\mathcal{L}=  -\Delta+(-\Delta)^{s},~ \text{for}~s \in  (0, 1).
 \end{equation}
The term ``mixed" describes an operator that combines local and nonlocal differential operators. In our case, the operator $\mathcal{L}$ in \eqref{1} is generated by the superposition of the classical Laplace operator $-\Delta$ and the fractional Laplace operator $(-\Delta)^{s}$ which is for a fixed parameter $s \in (0,1)$  defined by
$${(- \Delta)^{s}u(x)} = C_{n,s}~ P.V. \int_{\mathbb{R}^n} {\dfrac{u(x)-u(y)}{|x-y|^{n+2s}}} ~ dy. $$ 
The term ``P.V." stands for Cauchy's principal value, while $C_{n,s}$ is a normalizing constant whose explicit expression is given by
 $$C_{n,s}= \bigg( \int_{\mathbb{R}^{n}} {\frac{1-cos({\zeta}_{1})}{|\zeta|^{n+2s}}}~d \zeta \bigg)^{-1}.$$ 
In the literature, there are numerous definitions of nonlocal normal derivatives. We consider the one given in \cite{dipierro2017nonlocal} and defined for smooth functions $u$ as
\begin{equation}\label{normal}
\mathcal{N}_{s}u(x)= C_{n,s}\dint_{\Omega} \dfrac{u(x)-u(y)}{|x-y|^{n+2s}}\,dy, \qquad  x\in \mathbb{R}^n\setminus\bar{\Omega}.
\end{equation}

The study of mixed operators of the type $\mathcal{L}$ as in the problem \eqref{1} is motivated by several applications where such kind of operators are naturally generated, including the theory of optimal searching, biomathematics, and animal forging for which we refer to \cite{dipierro2022non, dipierro2021description, MR3771424}. In applied sciences, they are used for investigating the changes in physical phenomena that have both local and nonlocal effects. For instance, they are present in bi-modal power law distribution systems, see \cite{MR4225516}. Furthermore, they are present in models that are derived from the combination of two distinct scaled stochastic processes. We refer to  \cite{MR1640881} for a comprehensive explanation of this phenomenon. In recent years, there has been a significant amount of work investigating elliptic problems with mixed-type operator $\mathcal{L}$, which contain both local and nonlocal features. The analytical properties of elliptic and parabolic partial differential equations, as well as intergro-differential equations, relies significantly on  the spectrum of associated linearized problems.
In particular, the study of principal eigenvalues is essential in the investigation of non changing sign solutions to semi-linear problems as in \cite{biagi2022breziss} and in local bifurcation  phenomena as well as in stability analysis  (see \cite{berestycki2016definition}).
In this regard, the case of mixed operators, the Dirichlet (nonlinear) eigenvalue problem was studied by Rossi et al. in   \cite{pezzo2019eigenvalues}
 \begin{equation*} 
      \begin{cases}    
     -\Delta_p u-\Delta_{J,p} u = \lambda|u|^{p-2}u   \quad 
 \text{in }\Omega,\\ 
 u=0 ~~\text{in} ~~\mathbb{R}^n\setminus \Omega, \\
 \end{cases} 
 \end{equation*}
where $\Delta_{J,p}$ is a nonlocal operator. 
Authors showed that the first eigenvalues satisfy ${\lambda_1}^{\frac{1}{p}} \to \Lambda$ as $p \to \infty$, where $\Lambda$ is defined in terms of the geometry of $\Omega$. 
We also refer \cite{Faber, cowan2024principal, garain2023mixed}, where the eigenvalue problem with mixed operators of the type $\mathcal L$ was investigated. Very recently many problems involving the mixed operator $\mathcal{L}$ with the Neumann boundary condition have also been studied, for more details, we refer to \cite{MR4275496,MR2911421,MR3445279}. The nonlinear generalisation of $\mathcal L$ given by $-\Delta_p +(-\Delta_p)^s$ has also started gaining attention, relating to them, we quote \cite{garain2023mixed}. 
Dipierro et. al in \cite{MR4438596} was one the first among the others who consider mixed operator problems in the presence of classical as well as non-local  Neumann boundary conditions. Their recent article  discusses the spectral properties and the $L^\infty$
bounds associated with a mixed local and nonlocal problem,  in relation to some physical motivations arising from
population dynamics and mathematical biology. 
Recently, Biagi et al. \cite{biagi2022mixed} showed regularity results and maximum principle for the mixed local and nonlocal operators, we also refer to \cite{antonini2023global} and \cite{su2022regularity} for further regularity type results.

Next, we recall some eigenvalue problems in the presence of mixed boundary conditions.
Going back to \cite{denzler1998bounds, denzler1999windows}, Denzler et al. considered the following eigenvalue problem with mixed Dirichlet and Neumann boundary conditions
  
 \begin{equation}\label{local}
      \left\{\begin{split}     
     -\Delta u &=  \lambda_1(\mathcal{D}) u, \quad u>0 \quad 
 \text{in }\Omega,\\ 
 u &=0~~\text{in} ~~D,\\
 \frac{\partial u}{\partial \nu}&=0 ~~\text{in}~~ N,
  \end{split} \right.
 \end{equation}
 where they investigated the question of how the eigenvalue $\lambda_1(\mathcal{D})$ behaves when the sets are configured with Dirichlet (or Neumann) conditions.  Leonori et al. in \cite{MR3784437} analyze the nonlocal counterpart of \eqref{local} under the mixed boundary conditions. Due to the nonlocal behaviour of their problem, the sets $D$ and $N$ can be of infinite measures which is a striking difference from the local case.
 
 Taking into account the above literature, we are concerned in the present work with the eigenvalue problems involving mixed local and nonlocal operator $\mathcal L$ under the mixed Dirichlet and Neumann boundary conditions, which up to our knowledge has not been investigated previously. 
 In this regard, one of the main assumptions we impose is that $\mathcal{N}$ is bounded, whereas $\mathcal{D}$ can possess infinite Lebesgue measures, that are consistent with \cite{biagi2022mixed}. Firstly, we provide a functional setup that embeds a variational structure to our problem. Under this framework, our first main result is the existence of the first eigenvalue and corresponding positive eigenfunctions with their expected characteristics viz. principal eigenvalue, simple and strictly positive. We also establish boundedness and H\"older regularity of eigenfunctions.
 In view of the local bifurcation setting, our paper establishes also a strong maximum principle (Lemma \ref{strmx}) and a few other regularity results in the Appendix which are independent of interest. We also study some asymptotic behaviour of first eigenvalues with respect to the Dirichlet set, when Neumann sets dissipate and vice versa. According to to our boundary conditions, we show Theorem \ref{introd}  which adapts main results in \cite{MR3784437}. We underline that the corresponding sufficient counterpart of Theorem \ref{introd} is an open question.

In another segment of our article, we investigate bifurcation type results, build upon the strong maximum principle and H\"older regularity results, for an asymptotically linear problem inclined with $(P_\lambda)$.  In particular, we study the following problem
 \begin{equation*}\label{2} 
\left\{\begin{split} \mathcal{L}u\: &= \lambda h(u),~~u>0~ \text{in} ~\Omega, \\
      u&=0~~\text{in} ~~{U^c},\\
 \mathcal{N}_s(u)&=0 ~~\text{in} ~~{\mathcal{N}}, \\
 \frac{\partial u}{\partial \nu}&=0 ~~\text{in}~~ \partial \Omega \cap \overline{\mathcal{N}},
    \end{split} \right.\tag{$Q_\lambda$}
         \end{equation*}
 where $\lambda>0$ and $h$ is an asymptotically linear function that satisfies the following conditions: 
\begin{enumerate}\label{hcond}
{\item[(f1)] $h\in C^1(\mathbb{R}^+,\mathbb{R}^+)$,
\item[(f2)] $h(t)= \theta t+ f(t)$, where $\theta>0$ and $f: \mathbb{R}\to \mathbb{R^+}$ such that $|f(t)|\leq C$, for some $C>0$. 
\item[(f3)] There exist $a>0,$ $\lim\limits_{t \to 0^+}\frac{h(t)}{t}= a>0$.}
     \end{enumerate}
     \textbf{$(f1)_0$} {We extend the continuous function $h$ to whole $\mathbb{R}$ in such a way that 
 $ h(t)= 0$ for $t\leq 0.$ The symbol used to represent this extension will remain unchanged.}

     A pair $(\lambda,u)\in \mathbb{R}^+\times \mathcal{X}^{1,2}_{\mathcal{D}}(U)$ that satisfies \eqref{2} in the weak sense is referred to be a (weak) solution of \eqref{2} if 
\begin{equation}
    \int_{\Omega} \nabla u\cdot\nabla \varphi \,dx +\int_{Q} \frac{(u(x)-u(y))(\varphi(x)-\varphi(y))}{|x-y|^{n+2s}} dxdy  = \lambda\int_{\Omega} h(u)\varphi\,dx,\\
\end{equation}
for all $\varphi \in \mathcal{X}^{1,2}_{\mathcal{D}}(U).$
The pair $(\lambda,0)$ is then a trivial solutions for \eqref{2} and forms the trivial line of solutions as $\lambda$ varies in $\mathbb{R}$. 
{For problem \eqref{2}, we prove the existence of positive solutions branches (i.e. connected sets), both bifurcating from $0$ and from infinity, following the ideas developed in  \cite{ambrosetti1996multiplicity, arcoya2001bifurcation}. For more details about the classical bifurcation theory, we refer to \cite{crandall1971bifurcation} and \cite{rabinowitz1974variational}. We recall that $(\lambda_0,0)$ is a bifurcation point if a branch $\Gamma$ of nontrivial positive solutions $(\lambda,u)$ emanates from the trivial line of solutions at $\lambda_0$. We will study the behaviour of $\Gamma$, where $\Gamma=\Gamma_0$ is the closure of non-trivial solutions (maximal) connected set of \eqref{2} emanating from $(\lambda_0,0)$. We prove that $\lambda_0$ is the unique possible bifurcation point for $\eqref{2}$ and similarly $\lambda_{\infty}$ is the unique possible bifurcation from infinity generating the branch $\Gamma=\gamma_\infty$ , see Section 5 for details. Consequently, using the global bifurcation theorem given in \cite{Rabinowitz}, we show that continua $\Gamma_0$, $\Gamma_{\infty}$ are unbounded. Lastly, we remark that in particular when $h(s)= s+ s^p$, where $1<p<2^*-1$ and $h(s)= s-s^p,~ 1<p<\infty$, the same bifurcation theory is applicable (which we have established for \eqref{2}). 
}

The bifurcation approach helps us to address several kinds of classical problems, including the anti-maximum principle, multiplicity near resonance and Landesman–Lazer type existence results for resonant problems. For instance
Arcoya et al. in \cite{arcoya2001bifurcation} studied bifurcation theory for asymptotically linear problems involving the Laplace operator. Several other applications are provided in frame of integral equations and ordinary and partial differential equations, such as in \cite{MR0159197}, \cite{MR0241213}, \cite{Rabinowitz} and references therein.  Additionally, Chhetri et al. developed some bifurcation results for fractional Laplacian problems in \cite{MR4017463} and Ambrosetti et al. in \cite{ambrosetti1996multiplicity} studied the bifurcation of positive solutions for specific quasilinear eigenvalue problems. Colorado et al. in  \cite{colorado2004eigenvalues}   performed the analysis of eigenvalues, bifurcation and H\"older continuity of solutions with mixed boundary conditions. To the best of
our knowledge, no article has examined bifurcation phenomena associated with the asymptotically mixed linear problems of the form \eqref{2} under the mixed Dirichlet-Neumann boundary conditions, so far.

The rest of this article is arranged as follows: Section 2 provides the functional framework required to address the problem \eqref{1} and \eqref{2}. It presents the specific notion of solutions that are employed and introduces auxiliary main results.
Section 3 is devoted to establishing the existence of the principle eigenvalue, strong maximum principle and the boundedness ($L^{\infty}$ estimate)  of the eigenfunctions corresponding to  \eqref{1}. In section 4, we present the complementary behaviour of first eigenvalues and proofs of main results when  Dirichlet and Neumann sets dissipate. Finally, in Section 5, we prove bifurcation results, both from zero and from infinity to the problem \eqref{2}. 
Last but not least, Appendix contains regularity results, specifically global $C^{1,\beta}$ regularity (Theorem \ref{thm5.2}) for eigenfunctions that can be employed for a more general class of mixed operators type problems involving mixed boundary conditions.

 \section{Functional framework and main results }
In this section, we set our notations and formulated the functional framework for \eqref{1}, which are used throughout the paper.
For every $s\in (0,1)$, we recall the fractional Sobolev spaces
$${H^{s}(\mathbb{R}^n)} =  \Bigg\{ u \in L^{2}(\mathbb{R}^n):~~\frac{|u(x) - u(y)|}{|x - y|^{\frac{n}{2} + s}} \in L^{2}({\mathbb{R}^n}\times {\mathbb{R}^n)} \Bigg\} $$ which contain $H^1(\mathbb R^n)$. We assume that $\Omega \cup \mathcal N$ is bounded with a smooth boundary. 
The symbol $U$ and $\Omega_k$ are used throughout the article instead of $(\Omega \cup {\mathcal{N}} \cup (\partial\Omega\cap\overline{\mathcal{N}}))$ and $(\Omega \cup {\mathcal{N}_k} \cup (\partial\Omega\cap\overline{\mathcal{N}_k}))$ (respectively) for sake of clarity.
We define the function space $\mathcal{X}^{1,2}_{\mathcal{D}}(U)$ as 
\begin{align*}
    \mathcal{X}^{1,2}_{\mathcal{D}}(U)  = \{u\in H^1(\mathbb{R}^n) : ~u|_{U} \in H^1_0(U) ~\text{and}~ u \equiv 0~ a.e. ~\text{in}~ {U^c}\}.
\end{align*}
Let us define 
 $$ \eta(u)^2 =  ||\nabla{u}||^2_{L^{2}(\Omega)}+ [u]^2_{s},$$
for $u\in\mathcal{X}^{1,2}_{\mathcal{D}}(U)$, where  $[u]_s$ is the Gagliardo seminorm of $u$ defined by 
 $$[u]^2_{s} = ~ \bigg(\int_{Q} \frac{|u(x)-u(y)|^{2}}{|x-y|^{n+2s}} \, dx dy \bigg)$$ and $Q= \mathbb R^{2n}\setminus (\Omega^c\times \Omega^c)$.
The following Poincar\'e type inequality can be established following the arguments of Proposition 2.4 in \cite{MR4065090} and taking advantage of partial Dirichlet boundary conditions in $U^c$.
\begin{Proposition}\label{Poin} (Poincar\'e type inequality) There exists a constant $C=C(\Omega, n,s)>0$ such that
$$
\dint_{\Omega}| u|^2\,dx\leq C\bigg(\int_{\Omega} |\nabla u|^2\,dx+  \int_{Q} \dfrac{|u(x)-u(y)|^2}{|x-y|^{n+2s}}\,dxdy\bigg),
$$
for every  $u\in\mathcal{X}^{1,2}_{\mathcal{D}}(U)$, i.e. $\|u\|^2_{L^2(\Omega)}\leq C \eta(u)^2$. 
\end{Proposition}
As a consequence of Proposition \ref{Poin}, $\eta(\cdot)$ forms a norm on $\mathcal{X}^{1,2}_{\mathcal{D}}(U)$ and $\mathcal{X}^{1,2}_{\mathcal{D}}(U)$ is a Hilbert space with the  inner product associated with $\eta(\cdot)$, defined for any $u,v\in \mathcal{X}^{1,2}_{\mathcal{D}}(U)$  by
$$\langle{ u},{ v}\rangle = \int_{\Omega} \nabla u. \nabla{v} \,dx + \int_{Q} {\dfrac{(u(x)-u(y)) (v(x)-v(y))}{|x-y|^{n+2s}}} ~ dx dy. $$
Consequently, we have the integration by-parts formula given in the following proposition.
 \begin{Proposition}\label{P}
 For every $ u,v\in  C^\infty_0(U)$, it holds
\begin{align*}
    \int_{\Omega}v \mathcal{L} u \,dx  
    &= \int_{\Omega} \nabla u \cdot\nabla{v} \,dx +  \int_{Q} {\dfrac{(u(x)-u(y)) (v(x)-v(y))}{|x-y|^{n+2s}}} ~ dx dy\\
    &-  \int_{\partial \Omega\cap\overline{{\mathcal{N}}}} v {\frac{\partial u}{\partial \nu}}~ d{\sigma}-  \int_{{\mathcal{N}}} v {\mathcal{N}}_s u~ dx.
    \end{align*}
    where $\nu$ denotes the outward normal on $\partial\Omega$.
    \end{Proposition}
    \begin{proof}
        By directly using the integration by parts formula and the fact that $u,v \equiv 0$ a.e. in $\mathcal{D}\cup(\partial\Omega\cap\overline{\mathcal{D}})=U^c$, we can follow Lemma 3.3 of \cite{dipierro2017nonlocal}, to obtain the conclusion.
    \end{proof}
  \begin{Corollary}
Since $ C^\infty_0(U)$ is dense in $\mathcal{X}^{1,2}_{\mathcal{D}}(U)$, so Proposition \ref{P} still holds for functions in $\mathcal{X}^{1,2}_{\mathcal{D}}(U)$.   
\end{Corollary}
We now define the notion of weak solution to \eqref{1}.
\begin{Definition}\label{d1}
We say that $u\in \mathcal{X}^{1,2}_{\mathcal{D}}(U)$ is a weak solution to Problem \eqref{1} if 
\begin{equation}\label{dd2.2}
    \int_{\Omega} \nabla u\cdot\nabla \varphi \,dx +\int_{Q} \frac{(u(x)-u(y))(\varphi(x)-\varphi(y))}{|x-y|^{n+2s}} dxdy  = \lambda\int_{\Omega} u\varphi\,dx,
\end{equation}
for all $\varphi \in \mathcal{X}^{1,2}_{\mathcal{D}}(U).$
\end{Definition}

Consequently to $\mathcal{X}^{1,2}_{\mathcal{D}}(U)\hookrightarrow H^1(\mathbb R^n)$ and Sobolev embeddings, we infer the following embedding result:

\begin{remark}\label{r2.5}
   For $U$ is bounded (since $\Omega\cup\mathcal{N}$ is bounded) with smooth boundary, then we have compact embedding
    $$\mathcal{X}^{1,2}_{\mathcal{D}}(U)\hookrightarrow \hookrightarrow L^q_{loc}(\R^n)$$
   for $q\in [1,2^*)$ and continuous embedding for $q\in[1, 2^*].$
\end{remark}
Recalling $U$ is bounded and Proposition \ref{Poin}, we define $\lambda_1(\mathcal{D})$ as 
\begin{equation}\label{lambda}
\lambda_1(\mathcal{D})=\inf_{u\in \mathcal{X}^{1,2}_{\mathcal{D}}(U)\setminus \{0\}}  \frac{\int_{\Omega}|\nabla{u}|^2\,dx+ \int_{Q} \frac{|u(x)-u(y)|^{2}}{|x-y|^{n+2s}} \, dxdy} {\dint_{\Omega} | u|^2\,dx}\,.    
\end{equation}
\
Equivalently, we can write  $\lambda_1 (\mathcal{D}) $ as
$$\lambda_1(\mathcal{D})=\inf_{ u\in \mathcal{X}^{1,2}_{\mathcal{D}}(U)\setminus\{0\},\, \|u\|^2_{L^2(\Omega)}=1}  \bigg(\int_{\Omega}|\nabla{u}|^2\,dx+ \int_{Q} \frac{|u(x)-u(y)|^{2}}{|x-y|^{n+2s}} \, dx dy\bigg).$$
\
It is worth noting that the study of distinct arrangements of  $\mathcal{D}$ and $\mathcal{N}$ in \eqref{1} is involved in the behaviour of the associated eigenvalues in the context of mixed operators as opposed to the purely local and nonlocal contexts. In our conditions,  the boundary of $\Omega$ is replaced by  $\R^n\setminus \Omega$ and various situations of both sets $\mathcal{D}$ and $\mathcal{N}$ can occur. In particular, one should take into consideration that sets $\mathcal{D}$, $\mathcal{N}$  can have different sizes and shapes in the geometric sense, precisely, how far they are located with respect to $\Omega$. We remark that $\mathcal{N}$ is necessarily bounded whereas $\mathcal{D}$ has infinite Lebesgue measure. 

The outcome we first establish pertains to the account of how to arrange a sequence of domains $\{\mathcal{N}_k\}_{k\in \mathbb{N}}$, where the Neumann condition is specified, in order to demonstrate that the corresponding first eigenvalue approaches the one with the Dirichlet condition entirely on $\R^N\setminus \bar{\Omega}$.

Giving the following definition,
\begin{Definition}
     We say that $\Omega$ is an admissible domain if it is a $C^{1,1}$ domain.
\end{Definition} 
Our following results i.e. Theorem \ref{theoNeu} and Theorem \ref{introd} are related to the behaviour of the eigenvalues to problem \eqref{1}.

 \begin{Theorem}\label{theoNeu}
Suppose that $\Omega$ is an admissible domain  and consider for any $k$,  Dirichlet and Neumann open sets, $\mathcal{D}_k,\, \mathcal{N}_k\subset \R^n\setminus \bar\Omega$ such  that
\begin{equation}\label{nkdk}
\mathcal{D}_k\cap \mathcal{N}_k=\emptyset, \ \ \ \ \big | \mathbb R^n\setminus (\Omega\cup \mathcal{D}_k \cup \mathcal{N}_k)\big| =0.
\end{equation} 
If, additionally, the sequence  $\{\mathcal{N}_{k}\}_{k\geq 1}$ of sets associated with Neumann conditions satisfy the following- ~$\forall~ R>0$,
$ \lim\limits_{k\rightarrow\infty} |\mathcal{N}_{k}\cap B_{R}|=0$  and $\lim\limits_{k\to \infty} |\partial\Omega \cap\overline{\mathcal{N}_k}|=0$,
where  $B_R=\{x\in \mathbb{R}^n, |x|< R\}$  then $\displaystyle\lim\limits_{k\rightarrow \infty}\lambda_{1} (\mathcal{D}_k) =\lambda_{1} (\R^n\setminus \Omega)$.

\end{Theorem}
In fact, within a specific range of $ s$, we observe a similar outcome that guarantees the convergence of the sequence$\{\lambda_1 (\mathcal{D}_k)\}_k$ to zero.

\begin{Theorem}\label{introd}
If $0<s<1/2$, $\Omega\subseteq\R^n$  is an admissible domain. Suppose $\mathcal{D}_k,\, \mathcal{N}_k\subset \R^n\setminus \bar\Omega$ such that \begin{equation*}
\mathcal{D}_k\cap \mathcal{N}_k=\emptyset, \ \ \ \ \big | \mathbb R^n\setminus (\Omega\cup \mathcal{D}_k \cup \mathcal{N}_k)\big| =0.
\end{equation*} 
If, additionally, the sequence  $\{\mathcal{D}_{k}\}_{k\geq 1}$ of sets associated with Dirichlet conditions satisfy the following- ~$\forall~ R>0$,
$\lim _{k \rightarrow \infty}\left|\mathcal{D}_k \cap B_R\right|=0$ and
$\lim_{k\to \infty} |\overline {{\mathcal{D}_k}}\cap \partial\Omega|=0$   
then $\lim _{k \rightarrow \infty} \lambda_{1,k}=0$.
\end{Theorem}

Next, we consider problem \eqref{2} in the frame of the bifurcation setting.  The following result pertains to the existence of a bifurcation point that has an unbounded connected component $\Gamma_0\subset\Gamma$ of positive solutions. 

 We fix \begin{equation}\label{lem0}
    \lambda_0=\frac{\lambda_1(\mathcal{D})}{a} ~~\text{with}~~a>0. 
    \end{equation}

\begin{Theorem}\label{thm210}
If $(f1)$, $(f2)$, $(f3)$ holds and $h(0)=0$, $\lambda_0$ is the unique bifurcation point from zero for positive solutions of \ref{2}. More precisely, there exists an unbounded connected component $\Gamma_0\subset\Gamma$ of positive solutions to \ref{2} emanating from $(\lambda_0, 0)$ and $\lambda_0$ is the only value satisfying this property.    
\end{Theorem}

The second result pertains to the existence of a bifurcation point from infinity, along with an unbounded connected component $\Gamma_{\infty}\subset\Gamma$ of nontrivial solutions.
We fix $\lambda_{\infty}=\frac{\lambda_1(\mathcal{D})}{\theta}$, for $\theta>0$ defined in $(f2)$.
\begin{Theorem}\label{thm52} 
Under $(f1)$ and $(f2)$,  $\lambda_{\infty}$ is a unique bifurcation point from infinity for positive solutions of \ref{2}. More precisely, there exists an unbounded component $\Gamma_{\infty}\subset\Gamma$ of positive solutions of \ref{2} emanating from  $(\lambda_{\infty}, \infty)$ and $\lambda_\infty$ is the only value satisfying this property.
    
\end{Theorem}


\section{First eigenvalue and its features}
This section contains the proof of the existence of the first eigenvalue and its properties. Moreover, the behaviour of the first eigenvalue according to the shape and the size of the Dirichlet and Neumann boundary sets is discussed.

The following result is a version of the strong maximum principle for classical solutions. We shall prove this result by combining Bony's maximum principle, see \cite{garroni2002second} and a version of Hopf lemma, see  \cite{antonini2023global}. Together with regularity results, it is used to prove the existence of continua of solutions to $Q_\lambda$ in the positive cone of  $\mathcal{X}^{1,2}_{\mathcal{D}}(U)$ and to show that the first eigenvalue is principle and simple.

\begin{Lemma}\label{strmx}
    Let $0\leq u\in \mathcal{X}^{1,2}_{\mathcal{D}}(U) \cap C^{0,\beta}(\mathbb R^n)$, for some $\beta \in (0,1)$ satisfies
 \begin{equation*} 
\left\{\begin{split} \mathcal{L}u\: &\geq 0~~ \text{in} ~\Omega, \\
      u&=0~~\text{in} ~~{U^c},\\
 \mathcal{N}_s(u)&=0 ~~\text{in} ~~{\mathcal{N}}, \\
 \frac{\partial u}{\partial \nu}&=0 ~~\text{in}~~ \partial \Omega \cap \overline{\mathcal{N}},
    \end{split} \right.
         \end{equation*}
     then either $u\equiv0$ in U or $u>0$ in $U$.
\end{Lemma}

\begin{proof}
If $u=0$ in $U$ then we are done. Otherwise if $u\geq 0$ in $U$ and nontrivial then for any $x_0\in \Omega$, $u(x_0)\neq 0$. Indeed, assuming that $u(x_0)=0$ for some $x_0\in \Omega$ implies that there exists a point in $\Omega$ where the minimum is achieved, with value $0$ that is $u(x_0)\leq u(x), ~\forall~x\in \mathbb{R}^n$. Now, using the Bony maximum principle, see [Proposition 1.2.12 in \cite{garroni2002second}], we have $-\Delta u(x_0)\leq 0$. Thus, we find
\begin{align*}
    0\leq \mathcal{L}u(x_0) &= (-\Delta)u(x_0) + (-\Delta)^su(x_0)\\
    & \leq  C_{n,s} \int_{\mathbb{R}^n}\frac{u(x_0)-u(y)}{|x_0-y|^{n+2s}}\,dy = -C_{n,s} \int_{U}\frac{u(y)}{|x_0-y|^{n+2s}}\,dy \leq 0
\end{align*}
which implies $$ \int_{U}\frac{u(y)}{|x_0-y|^{n+2s}}\,dy=0.$$
Thus $u\equiv 0$ in $U$, which is a contradiction. Hence $u>0$ in $\Omega$. Now,  if $x\in \mathcal{N}$, then using the definition of $\mathcal{N}_s$ (see  \eqref{normal}) and $\mathcal{N}_s u(x)=0$, we have
$$ u(x)\int_{\Omega} \frac{dy}{|x-y|^{n+2s}}= \int_{\Omega} \frac{u(y) dy}{|x-y|^{n+2s}},$$
which implies
$$
u(x)=\frac{\int_{\Omega} \frac{u(y) dy}{|x-y|^{n+2s}}}{\int_{\Omega} \frac{dy}{|x-y|^{n+2s}}}>0.$$
Lastly let $x\in \partial\Omega\cap\partial\mathcal{N}$ then we have $\frac{\partial u}{\partial \nu}(x)\geq 0$. But using the version of Hopf Lemma in \cite{antonini2023global}, it is not hard to see that $\frac{\partial u}{\partial \nu}(x)<0$ which is a contradiction. 
Therefore we conclude that $u(x)>0$ in $\partial \Omega \cap \partial\mathcal{N}$ which completes the proof.
\end{proof}

We start by connecting  $\lambda_1(\mathcal{D})$ as defined in \eqref{lambda} with the first eigenvalue w.r.t. $(P_\lambda)$ in the following elementary result.
\begin{Proposition}\label{prop3.1}
$\lambda_1(\mathcal{D})$ is the first eigenvalue of \eqref{1}.
\end{Proposition}
\begin{proof}
Let $\{u_k\}_{k\geq 1}\in \mathcal{X}^{1,2}_{\mathcal{D}}(U)$ such that $\|u_k\|_{L^2{(\Omega)}}=1$,  be  a minimizing sequence associated to $\lambda_1(\mathcal{D})$ as defined in \eqref{lambda}.  Then we can infer 
\begin{equation}
\lim_{ k\to\infty}   \bigg(\int_{\Omega}|\nabla{u_k}|^2\,dx+ \int_{Q} \frac{|u_k(x)-u_k(y)|^{2}}{|x-y|^{n+2s}} \, dx dy\bigg)=\lambda_1(\mathcal{D}).
    \end{equation}
Then $\{u_k\}_{k\in\mathbb N}$ is bounded in $\mathcal{X}^{1,2}_{\mathcal{D}}(U)$. So there exists a $M>0$ such that $\eta(u_k)\leq M,~~ \forall~ k\in\mathbb{N}$. Since $\mathcal{X}^{1,2}_{\mathcal{D}}(U)$  is reflexive and from Remark \ref{r2.5}, 
we get up to an extraction of a subsequence that, for some $u\in \mathcal{X}^{1,2}_{\mathcal{D}}(U)$, 
$$u_k\rightharpoonup u \mbox{ in }\mathcal{X}^{1,2}_{\mathcal{D}}(U),\qquad u_k\rightarrow u \mbox{ in }L^2_{loc}(\mathbb{R}^n),\quad \text{and }~ u_k \to u ~\text{pointwise a.e. in}~ \mathbb\R^n, \text{as} ~k\to\infty.$$
Now, by using the weakly lower Semi-continuity, we have
$$
  \bigg(\int_{\Omega}|\nabla{u}|^2\,dx+ \int_{Q} \frac{|u(x)-u(y)|^{2}}{|x-y|^{n+2s}} \, dx dy\bigg)\leq \liminf_{k\to \infty}   \bigg(\int_{\Omega}|\nabla{u_k}|^2\,dx+ \int_{Q} \frac{|u_k(x)-u_k(y)|^{2}}{|x-y|^{n+2s}} \, dx dy\bigg)\leq \lambda_1(\mathcal{D}).
$$
We define  the functional $J\,: \mathcal{X}^{1,2}_{\mathcal{D}}(U)\mapsto \mathbb R$ defined  for any $v\in \mathcal{X}^{1,2}_{\mathcal{D}}(U)$ by
$$ J(v)= \int_{\Omega}|v|^2\,dx.$$ 
Again from Remark \ref{r2.5},  we have that $J(u)=1$ and $\lambda_1(\mathcal{D})=\eta(u)$.
Setting the constraint set
$ A=\{u\in\mathcal{X}^{1,2}_{\mathcal{D}}(U) : J(u)=1\}$, from the  definition of $\lambda_1(\mathcal{D})$ and using Lagrange multiplier rule, we infer that
$ \eta'( u)u=2\eta(u)=\lambda J'(u) u=2\lambda$, for some $ \lambda\in\R$. Thus,
we get $\lambda_1(\mathcal{D})=\lambda.$
\end{proof}

\begin{Lemma}\label{ll2.9}
  The first eigenvalue of $\mathcal L$ with mixed boundary conditions, as in \eqref{1}, is positive i.e. $\lambda_1(\mathcal{D})>0$.
     \end{Lemma}
\begin{proof}
This is a straightforward consequence of that $\lambda_1(\mathcal{D})$ is achieved as established in the proof of Proposition \ref{prop3.1}.
\end{proof}

We now recall a Picone-type inequality, whose proof can be seen in ({{\cite{mukherjee2023nonlocal}}}).
\begin{theorem}\label{T2.12}
Let $u,v \in \mathcal{X}^{1,2}_{\mathcal{D}}(U)$ and suppose that $\mathcal{L}u\geq 0$ is a bounded radon measure in $\Omega$,  $u>0$ in  $U$ and $\frac{\partial u}{\partial \nu}\geq 0$ on $\overline{\mathcal N}\cap \partial\Omega$, then 
\begin{equation}\label{eeA}{\int_{\overline{\mathcal{N}}\cap\partial\Omega}\frac{|v|^2}{u}\frac{\partial u}{\partial \nu} \,d{\sigma}+\int_{\mathcal{N}}\frac{|v|^2}{u} \mathcal{N}_s u ~dx + \int_{\Omega}\frac{|v|^2}{u} \mathcal{L} u ~dx \leq \eta(v)^2 }.
   \end{equation}
 \end{theorem}
Let us now state and prove some well-known expected properties of first eigenvalues. 
\begin{Proposition}\label{prop-1}
The first smallest eigenvalue $\lambda_1({\mathcal{D}})$(obtained in Lemma \ref{ll2.9}) satisfies the following :
    \begin{enumerate}
    \item  First eigenfunctions are bounded, i.e. lies in $L^{\infty}(U).$ 
     \item Any eigenfunction, $\varphi$, associated $\lambda_1(\mathcal{D})$  do not change sign. Precisely, either $\varphi>0$ in $U$ or $\varphi<0$ in $U$, i.e. $\lambda_1(\mathcal{D})$ is a principal eigenvalue.
        \item  $\lambda_1(\mathcal{D})$ is simple.
        \item for any eigenvalue $\lambda>\lambda_1({\mathcal{D}})$, the associated eigenfunctions are  sign changing in $U.$
        \end{enumerate}
\end{Proposition}
\begin{proof}

Assertion 1.
Let us fix $u$ as the first eigenfunction associated with $\lambda_1(\mathcal{D})$ such that $\|u\|_{L^2{({  \Omega}})}=1$.
For $\rho>0$, we set
$
\hat{u}=\sqrt{\rho} u.
$
Now defining $d_k=1-\frac{1}{2^k}, ~\forall~k \in \mathbb{N}$ and
$$
v_k=\hat{u}-d_k, \quad w_k=\left(v_k\right)_{+}=\max \left\{v_k, 0\right\}, \quad \mathcal{U}_k=\left\|w_k\right\|_{L^2({  \Omega})}^2,
$$ we conclude that 
$\left\|\hat{u}\right\|_{L^2(\Omega)}^2=\rho\left\|u\right\|_{L^2(\Omega)}^2=\rho$(since $\left\|u\right\|_{L^2(\Omega)}=1$) and $v_k \geq v_{k+1}$, $w_k \geq w_{k+1}$(since $d_k<d_{k+1}$).
If $u \in \mathcal{X}^{1,2}_{\mathcal{D}}(U)$, by the definition of  $\mathcal{X}^{1,2}_{\mathcal{D}} (U), u\in H^1\left(\mathbb{R}^n\right)$.  Moreover, since $\hat{u}=\sqrt{\rho} u \equiv 0$ a.e. in $U^c$, one also has
$$
v_k=\hat{u}-d_k=-d_k<0 ~ \text{ on } U^c ~\text{and}~  w_k \in \mathcal{X}^{1,2}_{\mathcal{D}}(U).
$$
 We use  $w_k$ as a test function in \eqref{dd2.2} to obtain
\begin{equation}\label{eq3.0.4}
\int_{\Omega}\nabla \hat{u}\cdot \nabla w_k d x  +\int_{Q} \frac{\left(\hat{u}(x)-\hat{u}(y)\right)\left(w_k(x)-w_k(y)\right)}{|x-y|^{n+2 s}} d x d y 
=\lambda \int_{\Omega} \hat{u} w_k d x.
    \end{equation}
Moreover, by the definition of $w_k$, we have
\begin{equation}\label{eq3.0.5}
\int_{\Omega}\nabla \hat{u}\cdot\nabla w_k d x=\int_{\Omega \cap\left\{\hat{u}>d_k\right\}}\nabla \hat{u}\cdot \nabla v_k \,d x=\int_{\Omega}\left|\nabla w_k(x)\right|^2 d x.
    \end{equation}
 Hence, by non negativity of second integral in \eqref{eq3.0.4} and from \eqref{eq3.0.5}, we deduce the following
\begin{equation}\label{eqqw}
\int_{\Omega}\left|\nabla w_k(x)\right|^2 d x \leq \lambda \int_{\Omega} \hat{u} w_k d x.
    \end{equation}
So,  using the Sobolev inequality (see  Theorem 2.4.1 in \cite{Kesavan}), we have
\begin{equation}\label{eq3.0.6}
\left(\int_{{  \Omega}}|w_k(x)|^{2^*}\,dx\right)^\frac{2}{2^*}\leq \mathcal{C} \int_{{  \Omega}}|\nabla w_k(x)|^2\,dx,~\text{for some}~ \mathcal{C}>0.
\end{equation}
 Suppose $x\in \left\{w_k>0\right\}$ then  we obtain
\begin{equation}\label{u_o}
\hat{u}<2^k w_{k-1}, ~\forall~k\geq 1
    \end{equation}
    and also as a result (see Theorem 3.2 in \cite{franzina2013fractional}),
\begin{equation}\label{eq3.0.9}
\left\{w_k>0\right\}=\left\{\hat{u}>d_k\right\} \subseteq\left\{\frac{1}{2^k}< w_{k-1}\right\}, ~\forall~k \geq 1.
\end{equation}
Now, from \eqref{eqqw} and using  \eqref{u_o}, $w_{k-1}\geq w_k$, we have
\begin{equation}\label{a.e.}
\begin{aligned}
\int_{\Omega}\left|\nabla w_k(x)\right|^2 d x & \leq \lambda \int_{\left\{w_k>0\right\}} \hat{u} w_k d x  \leq \lambda 2^k  \int_{\left\{w_k>0\right\}} w_{k-1} w_k d x \\
& \leq \lambda 2^k   \int_{{  \Omega}} w_{k-1}^2 d x= \lambda 2^k  \left\|w_{k-1}\right\|_{L^2({  \Omega})}^2  = \lambda 2^k   \mathcal{U}_{k-1} .
\end{aligned}
    \end{equation}
As a consequence, using \eqref{eq3.0.9} in \eqref{a.e.} and by the Chebyshev inequality, we obtain
\begin{equation}\label{new-tm1}
\begin{aligned}
\mathcal{U}_{k-1} & =\int_{{  \Omega}} w_{k-1}^2 d x \geq \int_{\left\{w_{k-1}>\frac{1}{2^{k}}\right\}} w_{k-1}^2 d x \geq \frac{1}{2^{2k}}\left|\left\{w_{k-1}>\frac{1}{2^k}\right\}\right| \geq \frac{1}{2^{2k}}\left|\left\{w_k>0\right\}\right| .
\end{aligned}
    \end{equation}
Now, we use the Hölder inequality, \eqref{eq3.0.6}, \eqref{a.e.} and \eqref{new-tm1} to obtain the following estimate
\begin{equation}\label{eqh}
\begin{aligned}
\mathcal{U}_k  =\left\|w_k\right\|_{L^2({  \Omega})}^2 \leq \left(\int_{{  \Omega}} |w_k|^{2^*} d x\right)^{\frac{2}{2^*}}\left|\left\{w_k>0\right\}\right|^{\frac{2}{n}}  
&\leq \mathcal{C}  \left(\int_{{  \Omega}}\left|\nabla w_k\right|^2 dx\right) \left|\left\{w_k>0\right\}\right|^{\frac{2}{n}} \\
& \leq \mathcal{C}\left(\lambda 2^k  \mathcal{U}_{k-1}\right)\left(2^{2 k} \mathcal{U}_{k-1}\right)^{\frac{2}{n}}\\ &=\mathbf{c}^{\prime}\left(2^{1+\frac{4}{n}}\right)^{k-1} \mathcal{U}_{k-1}^{1+\frac{2}{n}},
\end{aligned}
    \end{equation}
where $\mathcal{C}>0$ is a Sobolev constant and $\mathbf{c}^{\prime}= \lambda 2^{1+\frac{4}{n}} \mathcal{C} $.
We can see  $\frac{2}{n}>0$,
$
r=2^{1+\frac{4}{n}}>1,
$
using \cite[Lemma 7.1, p.220]{MR1962933} 
that $\mathcal{U}_k \rightarrow 0$ as $k \rightarrow \infty$, provided that
$$
\mathcal{U}_0=\left\|\hat{u}\right\|_{L^2({  \Omega})}^2=\rho<\left(\mathbf{c}^{\prime}\right)^{\frac{-n}{2} } r^{\frac{-n^2}{4}} .
$$
As a consequence, if $\rho>0$ is small enough, we can use Dominated Convergence Theorem to conclude
$$
0=\lim _{k \rightarrow \infty} \mathcal{U}_k=\lim _{k \rightarrow \infty} \int_{{  \Omega}}\left(\hat{u}-d_k\right)_{+}^2 d x=\int_{{  \Omega}}\left(\hat{u}-1\right)_{+}^2 d x .
$$
Recalling $\hat{u}=\sqrt{\rho} u$ and $u \geq 0$, using above we obtain
$0 \leq u \leq \frac{1}{\sqrt{\rho}} \text{ a.e. in } {  \Omega}$,
which implies $u \in L^{\infty}({  \Omega})$.
{  Now,  if $x\in \mathcal{N}$, then using the definition of $\mathcal{N}_s$ (see  \eqref{normal}) and $\mathcal{N}_s u(x)=0$, we have
$$ u(x)\int_{\Omega} \frac{dy}{|x-y|^{n+2s}}= \int_{\Omega} \frac{u(y) dy}{|x-y|^{n+2s}},$$
which implies
$$
u(x)=\frac{\int_{\Omega} \frac{u(y) dy}{|x-y|^{n+2s}}}{\int_{\Omega} \frac{dy}{|x-y|^{n+2s}}}.$$
Since $u\in L^{\infty}(\Omega)$ then we get $|u(x)|\leq \|u\|_{L^{\infty}(\Omega)}$, for each $x\in\mathcal{N}.$ Thus, we conclude that $u\in L^{\infty}(U).$
}

Assertion 2. Here, we only give some ideas towards the proof of this Proposition, for details one can refer to [Proposition 5.1 in \cite{biagi2021global}].
First we note that $\eta$ is $C^1(\mathcal{X}^{1,2}_{\mathcal{D}}(U), \R)$ and 
$$
\mathcal{M}:=\left\{u \in \mathcal{X}^{1,2}_{\mathcal{D}}(U): J(u)=\int_{\Omega}|u|^2 d x=1\right\}
$$
is a $C^1$ - Banach manifold.
From above, we have that
\begin{equation}\label{e}
\lambda_1\left(\mathcal{D}\right):=\inf \{\mathcal{I}(u): u \in \mathcal{M}\} 
    \end{equation}
    and is achieved on some $u$. Furthermore, since $\eta(|u|)\leq \eta(u)$, $|u|$ is also a minimizer.
Applying {   Lemma \ref{Ns} and Lemma \ref{strmx}} 
 to $|u|$, we get that $|u|>0$ in $U$ which implies assertion 2.\\

Assertion 3. Now, we prove that  $\lambda_1(\mathcal{D})$ is simple, i.e. if $\varphi_1,\;\varphi_2$ are two eigenfunctions corresponding to $\lambda_1(\mathcal{D})$ 
then $\varphi_1=\alpha \varphi_2$, with $\alpha\in \mathbb{R}.$
W.l.o.g we assume that eigenfunction $\varphi_1\in \mathcal{X}^{1,2}_{\mathcal{D}}(U)$ associated to $\lambda_1\left(\mathcal{D}\right)$ is non negative and normalized. Let us suppose that $\varphi_2 \in \mathcal{X}^{1,2}_{\mathcal{D}}(U)$ is  eigenfunction satisfying $\varphi_2 \not \equiv \varphi_1$, associated to $\lambda_1\left(\mathcal{D}\right)$.  We may suppose that $\varphi_2 \not \equiv 0$, otherwise we are done.
From Assertion 2, we know that either $\varphi_2 > 0$ or $\varphi_2 <0$  in $U$. Let us consider the case
\begin{equation}\label{ei2}
\varphi_2 < 0 \text { in } U,
    \end{equation}
the other being analogous. We define
$$
\tilde{\varphi_2}=\frac{\varphi_2}{\left\|{\varphi}_2\right\|_{L^2(U)}} \text { and } h_1=\varphi_1- \tilde{\varphi}_2.
$$
So we aim to show that
\begin{equation}\label{g10}
h_1(x)=0 \text { a.e. } x \in \mathbb{R}^n .
    \end{equation}
It is easy to observe that $h_1$ is also an eigenfunction relative to $\lambda_1(\mathcal{D})$ and by assertion 2 again, $h_1 \geq 0$ or $h_1 \leq 0$ a.e. in $U$. Thus either $\varphi_1 \geq \tilde{\varphi}_2$ or $\varphi_1 \leq \tilde{\varphi}_2$ in $U$ which implies using \eqref{ei2} and the non-negativity of $\varphi_1$,
\begin{equation}\label{uv}
\text { either } \varphi_1^2 \geq \tilde{\varphi}_2^2 \text { or } \varphi_1^2 \leq \tilde{\varphi}_2^2 \text { a.e. in } U \text {. }
    \end{equation}
On the other hand,
\begin{equation}\label{int12}
\int_{U}\left(\varphi_1^2(x)-\tilde{\varphi}_2^2(x)\right) d x=\left\|\varphi_1\right\|_{L^2(U)}^2-\left\|\tilde{\varphi}_2\right\|_{L^2(U)}^2=1-1=0,
    \end{equation}
since($\left\|\varphi_1\right\|_{L^2(U)}^2=1,$ from the proof of assertion 2). 
Thus, \eqref{int12} and \eqref{uv} gives that $\varphi_1^2-\tilde{\varphi}_2^2=0$ and hence $\varphi_1=\tilde{\varphi}_2$, so $h_1=0$ a.e. in $U$. Since $h_1$ vanishes outside $U$, we find $h_1=0$ a.e. in $\mathbb{R}^n$, that is our claim \eqref{g10}.
Then, as a consequence of \eqref{g10}, we conclude that
$h_1=\varphi_1-\tilde{\varphi_2}=0$, which implies
$$\varphi_1= \frac{\varphi_2}{\left\|{\varphi}_2\right\|_{L^2(U)}}.$$
Hence, $\varphi_2$ is proportional to $\varphi_1$, and this proves assertion 3.

Assertion 4. Now, if $\lambda>\lambda_1\left(\mathcal{D}\right)$ is an eigenvalue of $\mathcal{L}$ with the mixed boundary conditions as in \eqref{1} and its corresponding eigenfunction is $u_{\lambda} \in \mathcal{X}^{1,2}_{\mathcal{D}}(U)$ such that $\|u_{\lambda}\|_{L^2{(\R^n)}}=1$. We claim that $u_{\lambda}$ is sign-changing. By contrast, we may assume that $u_{\lambda}$ has a constant sign, say $u_{\lambda} \geq 0  \text{ a.e. in}~ \R^n$. By {   Lemma \ref{Ns} and Lemma \ref{strmx}}, we have $u_{\lambda}>0$ in $U$. 
 
Suppose $\varphi_2 \in \mathcal{X}^{1,2}_{\mathcal{D}}(U)$ is another positive eigenfunction associated to $\lambda_1\left(\mathcal{D}\right)$.   For $\varepsilon>0$, we define $\varphi_1=|u|$, 
$$
 \quad u_{\varepsilon}=\frac{\varphi_2^2}{\left(\varphi_1+\varepsilon\right)}.
$$
 $\varphi_{2}$, $u_{\varepsilon} \in \mathcal{X}^{1,2}_{\mathcal{D}}(U)$.
 Suppose $\varphi_2 \in \mathcal{X}^{1,2}_{\mathcal{D}}(U)$ is another positive eigenfunction associated to $\lambda_1\left(\mathcal{D}\right)$.   For $\varepsilon>0$, we define $\varphi_1=|u|$  and 
$$
\varphi_{2, \varepsilon}=\min \left\{\varphi_2, \frac{1}{\varepsilon}\right\} \quad \text { and } \quad u_{\varepsilon}=\frac{\varphi_{2, \varepsilon}^2}{\varphi_1+\varepsilon} .
$$ From Assertion 1, $\varphi_{2,\varepsilon}$, $u_{\varepsilon} \in \mathcal{X}^{1,2}_{\mathcal{D}}(U)$. 
 We  use $u_{\varepsilon}$ as test function in \eqref{dd2.2} solved by $\varphi_1$ then 

\begin{equation}\label{3.0.13}
\int_{\Omega}\nabla \varphi_1\cdot \nabla u_{\varepsilon} d x 
+\int_{Q} \frac{\left(\left(\varphi_1+\varepsilon\right)(x)-\left(\varphi_1+\varepsilon\right)(y)\right)\left(u_{\varepsilon}(x)-u_{\varepsilon}(y)\right)}{|x-y|^{n+2 s}} d x d y 
=\lambda_1(\mathcal{D}) \int_{\Omega} \varphi_1 u_{\varepsilon} d x.
\end{equation}

By the discrete Picone inequality (see Theorem 18 of \cite{MR3393266}), we have
$$
\left(\left(\varphi_1+\varepsilon\right)(x)-\left(\varphi_1+\varepsilon\right)(y)\right)\left(u_{\varepsilon}(x)-u_{\varepsilon}(y)\right) \leq\left|\varphi_{2, \varepsilon}(x)-\varphi_{2, \varepsilon}(y)\right|^2.
$$
Additionally, knowing that the map  $h \mapsto \min \{|h|, 1 / \varepsilon\}$ is Lipschitz then we obtain that
$$
\left(\left(\varphi_1+\varepsilon\right)(x)-\left(\varphi_1+\varepsilon\right)(y)\right)\left(u_{\varepsilon}(x)-u_{\varepsilon}(y)\right) \leq\left|\varphi_2(x)-\varphi_2(y)\right|^2.
$$
Now, consider the function
$$
\mathcal{G}\left(\varphi_{2, \varepsilon}, \varphi_1+\varepsilon\right)=\left|\nabla \varphi_{2, \varepsilon}\right|^2-\nabla \varphi_1\cdot \nabla u_{\varepsilon}.
$$
As a consequence of the  Picone identity, see [Proposition 9.61 in \cite{MR3136201}] and also see \cite{MR1618334}, we have that $\mathcal{G}\left(\varphi_{2, \varepsilon}, \varphi_1+\varepsilon\right) \geq 0$, that is 
\begin{equation}\label{picon}
\begin{aligned}
&\left|\nabla \varphi_{2, \varepsilon}\right|^2-\nabla \varphi_1. \nabla u_{\varepsilon}\geq 0~~\text{or},\\
&\nabla \varphi_1 \cdot\nabla u_{\varepsilon} \leq\left|\nabla \varphi_{2, \varepsilon}\right|^2 \leq\left|\nabla \varphi_2\right|^2.
    \end{aligned}
    \end{equation}
Based on the information gathered, the limit can be determined as $\varepsilon \rightarrow 0$ in \eqref{3.0.13}: by employing the Dominated Convergence theorem on the left-hand side of \eqref{3.0.13} and {the Fatou's lemma on the right-hand side}, we obtain that
\begin{equation}\label{domi}
\begin{aligned}
& \int_{\Omega}\nabla \varphi_1\cdot \nabla\left(\frac{\varphi_2^2}{\varphi_1}\right) d x +\int_{Q} \frac{\left(\varphi_1(x)-\varphi_1(y)\right)}{|x-y|^{n+2 s}}\left(\frac{\varphi_2^2(x)}{\varphi_1(x)}-\frac{\varphi_2^2(y)}{\varphi_1(y)}\right) d x d y \\
&\geq \lambda_1(\mathcal{D}) \int_{\Omega} \varphi_2^2 d x =\int_{\Omega}\left|\nabla \varphi_2\right|^2 d x+\int_{Q} \frac{\left|\varphi_2(x)-\varphi_2(y)\right|^2}{|x-y|^{n+2 s}} d x d y.
\end{aligned}
    \end{equation}
On the other hand, recalling \eqref{picon} then 
we have the estimate
\begin{equation}\label{3.0.15}
\begin{aligned}
\int_{\Omega}\nabla \varphi_1\cdot \nabla\left(\frac{\varphi_2^2}{\varphi_1}\right) d x 
+\int_{Q} \frac{\left(\varphi_1(x)-\varphi_1(y)\right)}{|x-y|^{n+2 s}}\left(\frac{\varphi_2^2(x)}{\varphi_1(x)}-\frac{\varphi_2^2(y)}{\varphi_1(y)}\right) d x d y \\
\quad \leq \int_{\Omega}\left|\nabla \varphi_2\right|^2 d x+\int_{Q} \frac{\left|\varphi_2(x)-\varphi_2(y)\right|^2}{|x-y|^{n+2 s}} d x d y .
\end{aligned}
    \end{equation}
Recalling  $u_{\lambda}>0$ in $U$,
and  equation solved by $u_{\lambda}$, using
$w_\epsilon=\frac{\varphi_{1}^2}{u_{\lambda}+\varepsilon}$
 in $\mathcal{X}^{1,2}_{\mathcal{D}}(U)$ as a test function, and from the above arguments as equations \eqref{domi} and \eqref{3.0.15}, we conclude that   
 $$ \int_{\Omega} |\nabla \varphi_1(x)|^2 \,dx +\int_{Q} \frac{|\varphi_1(x)-\varphi_1(y))|^2}{|x-y|^{n+2s}} dxdy  = \lambda.$$
 Since $\varphi_1$ is a solution to \eqref{1} corresponding eigenvalue $\lambda_1(\mathcal{D})$, then we conclude $\lambda=\lambda_1({\mathcal{D)}}$  that gives a contradiction with our assumption $\lambda>\lambda_1({\mathcal{D)}}$.
Thus, $u_{\lambda}$ is  sign-changing in $U$. 
\end{proof}

\begin{Lemma}\label{reg-1}
    Suppose $u$ is an eigenfunction of $\mathcal L$ then $u\in C^{0,\beta}(\mathbb{R}^n),$ for some $\beta\in (0,1).$
\end{Lemma}
\begin{proof}
   It is clear that one can repeat the proof of Proposition \ref{prop-1}(1) to obtain that every eigenfunction is bounded, in particular, $u\in L^{\infty}(U)$. Now, using Theorem \ref{thm5.2} and Lemma \ref{Ns} (see Appendix), we can conclude that  $u\in C^{0,\beta}(\mathbb{R}^n).$
\end{proof}

The following reveals the orthogonality of eigenfunctions corresponding to distinct eigenvalues.
\begin{Lemma}
    Let  $\phi_1$, $\phi_2\in \mathcal{X}^{1,2}_{\mathcal{D}}(U)$ be eigenfunctions corresponding to two different eigenvalues $\mu_1 \neq {\mu_2}$ respectively w.r.t. \eqref{1}, then
$$
\langle \phi_1, {\phi_2}\rangle_{\mathcal{X}^{1,2}_{\mathcal{D}}(U)}=0=\int_{\Omega} \phi_1(x) {\phi_2}(x) d x .
$$
\end{Lemma} 
\begin{proof}
Suppose  $\phi_1 \not \equiv 0$ and ${\phi_2} \not \equiv 0$ a.e. in $\Omega$ and we set $f:=\phi_1 /\|\phi_1\|_{L^2(\Omega)}$ and ${g}:={\phi_2} /\|{\phi_2}\|_{L^2(\Omega)}$, which are eigenfunctions to eigenvalues $\mu_1$ and $\mu_2$ respectively. Testing $(P_{\mu_1})$ with  ${g}$ as test function and $(P_{\mu_2})$ with $f$ as test function, we obtain
\begin{equation}\label{eq2.16}
 \begin{split}
 \int_{\Omega}\nabla f(x)\cdot\nabla g(x)\,dx+\int_{Q} \frac{(f(x)-f(y))(g(x)-{g}(y))}{|x-y|^{n+2s}} \,dx \,dy 
&=\mu_1 \int_{\Omega} f(x) {g}(x) d x\\
&={\mu_2} \int_{\Omega} f(x) {g}(x) d x,
 \end{split}   
\end{equation}
\
that is
$$
(\mu_1-{\mu_2}) \int_{\Omega} f(x) {g}(x) d x=0 .
$$
 Since $\mu_1 \neq {\mu_2}$, then
\begin{equation}\label{eq2.17}
\int_{\Omega} f(x) {g}(x) d x=0
\end{equation}
By plugging \eqref{eq2.17} into \eqref{eq2.16}, we deduce that
$$
\langle f, {g}\rangle_{\mathcal{X}^{1,2}_{\mathcal{D}}(U)}=\int_{\Omega}\nabla f(x)\cdot\nabla g(x)\,dx+\int_{Q}\frac{(f(x)-f(y))({g}(x)-{g}(y))} {|x-y|^{N+2s}} d x d y=0
$$
which completes the proof.
\end{proof}

\section{Complementary behavior of first eigenvalues}\label{3}

Let us have a look at the following sequence of eigenvalue problems
\begin{equation}\label{sequence} 
      \left\{\begin{split}    
     \mathcal{L} u_{1,k}& =  \lambda_{1,k}  u_{1,k}, \quad  u_{1,k}>0 \quad 
 \text{in }\Omega,\\ 
   u_{1,k}&=0~~\text{in} ~~\Omega_k^c,\\
 \mathcal{N}_s u_{1,k}&=0 ~~\text{in} ~~{{\mathcal{N}_k}}, \\
 \frac{\partial u_{1,k}}{\partial \nu}&  =0 ~~\text{in}~~ \partial \Omega \cap \overline{{\mathcal{N}_k}},
  \end{split} \right.
 \end{equation}
where $\mathcal{D}_k,\,\mathcal{N}_k\subset\R^n\setminus\bar{\Omega}$ satisfies \eqref{nkdk}
and  $\lambda_{1,k} = \lambda_1 (\mathcal{D}_k) $ with $u_{1,k}$ representing  the  corresponding positive and normalized (in $L^2(\R^n)$  that is $\int_{\R^n} |u_{1,k}|^2\,dx=1$)  eigenfunction. The next result  deals with $\lambda_1(\mathcal D)$, when $\mathcal D =\emptyset$.
\begin{theorem}\label{phi}
   The first eigenvalue for $\mathcal L$ under the Neumann boundary condition only is zero, i.e.  $\lambda_1(\emptyset)=0$.
\end{theorem}
\begin{proof}
  From \cite{mugnai2022mixed}, we have $\lambda_1(\emptyset)=0$.   
\end{proof}

Let us define $\lambda_1(\R^n \setminus \Omega)$ as the first eigenvalue of $\mathcal L$ with Dirichlet boundary condition in $\mathbb R^N\setminus\Omega$, (see (4.8)  of  \cite{biagi2022breziss}). 
First, we establish qualitative properties of solutions to \eqref{sequence}, given by $u_{1,k}$.
\begin{Proposition} \label{u1k}
Let $\Omega$ be an admissible domain and  pairs of sets $\mathcal{D}_k$ and $\mathcal{N}_k$ satisfy \eqref{nkdk} with $\Omega_k$ bounded. Then,
there exists a $u_{1,k} \in \mathcal{X}^{1,2}_{\mathcal{D}_k}(\Omega_k)$ satisfying \eqref{sequence}.
Moreover:
\begin{itemize}
\item[(1)] $\lambda_1 (\emptyset) = 0 <\lambda_1 (\mathcal{D}_k)=\eta\left(\frac{u_{1,k}}{\|u_{1,k}\|_{L^2(\Omega)}}\right)\leq \lambda_1 (\R^n\setminus \Omega)$;
\item[(2)] $u_{1,k}\geq 0$ in $\mathbb{R}^n$ and  $u_{1,k}> 0$ in $\Omega_k $;
\item[(3)]  $\{u_{1,k}\}_{k\geq 1}$ is uniformly bounded in $L^\infty(\Omega_k ).$ 
\end{itemize}
\end{Proposition}
\begin{proof}
(1) From Theorem \ref{phi}, one has $\lambda_1(\emptyset)$. The existence of the pair of eigenvalues and eigenfunctions i.e. $( \lambda_{1,_k}, u_{1,k} ) \in \R^+ \times\mathcal{X}^{1,2}_{\mathcal{D}_k}(\Omega_k) $ follows from Lemma \ref{ll2.9}.
Moreover,  assertion (1) is a consequence of the following
\begin{equation}\label{R}
\lambda_{1,k}=\inf_{0\not\equiv u\in \mathcal{X}^{1,2}_{\mathcal{D}_k}(\Omega_k), \, \|u\|^2_{L^2(\Omega)}=1}  \bigg(\int_{\Omega}|\nabla{u}|^2\,dx+ \int_{Q} \frac{|u(x)-u(y)|^{2}}{|x-y|^{n+2s}} \, dx dy\bigg),
\end{equation}
and Proposition \ref{Poin}  with{
\begin{equation}\label{inclus}
 \mathcal{X}^{1,2}_{\mathcal{D}_k}(\Omega_k)\subset H^1(\R^n).
\end{equation}}
(2) It can be followed by the attainability of the first eigenvalue.   There exists $u_{1,k}\in \mathcal{X}^{1,2}_{\mathcal{D}_k}(\Omega_k)$ for each $k$ which is minimizer of the Rayleigh quotient $\lambda_{1,k}$ in \eqref{R}. But  $|u_{1,k}|$ also minimizes the quotient, so we can assume $u_{1,k}\geq 0$.
By  {   Lemma \ref{Ns} and Lemma \ref{strmx}},
we have $u_{1,k}>0$ in $\Omega_k$. 

(3)
  Let us suppose for any $m\in \mathbb{N}$, $u_{1,m}$ as the first eigenfunction such that $\|u_{1,m}\|_{L^2{(  \Omega})}=1$.
For $\rho>0$, we set
$
\hat{u}_{1,m}=\sqrt{\rho} u_{1,m}.
$
Now defining $d_k=1-\frac{1}{2^k}$ and
$$
v_{k,m}=\hat{u}_{1,m}-d_k, \quad w_{k,m}=\left(v_{k,m}\right)_{+}=\max \left\{v_{k,m}, 0\right\}, \quad \mathcal{U}_{k,m}=\left\|w_{k,m}\right\|_{L^2({  \Omega})}^2, ~\forall~k,m \in \mathbb{N}
$$ 
we conclude that 
$\left\|\hat{u}_{1,m}\right\|_{L^2({  \Omega})}^2=\rho\left\|u_{1,m}\right\|_{L^2({  \Omega})}^2=\rho$ (since $\left\|u_{1,m}\right\|_{L^2({  \Omega})}=1$) and $v_{k,m} \geq v_{k+1,m}$, $w_{k,m} \geq w_{k+1,m}$(since $d_{k}<d_{k+1}$).
If $u_{1,m} \in \mathcal{X}^{1,2}_{\mathcal{D}_m}(\Omega_m)$, by the definition of  $\mathcal{X}^{1,2}_{\mathcal{D}_m} (\Omega_m), u_{1,m}\in H^1\left(\mathbb{R}^n\right)$ (where $\Omega_m= (\Omega \cup {\mathcal{N}_m} \cup (\partial\Omega\cap\overline{\mathcal{N}_m}))$.  Moreover, since $\hat{u}_{1,m}=\sqrt{\rho} u_{1,m} \equiv 0$ a.e. in $\Omega_m^c$, one also has
$$
v_{k,m}=\hat{u}_{1,m}-d_{k}=-d_{k}<0 \text{ on } \Omega_m^c ~\text{and}~  w_{k,m} \in \mathcal{X}^{1,2}_{\mathcal{D}_m}(\Omega_m).
$$
 We use  $w_{k,m}$ as a test function in \eqref{dd2.2} to obtain
\begin{equation}\label{eq3.0.4t}
\int_{\Omega}\nabla \hat{u}_{1,m}\cdot \nabla w_{k,m} d x  +\int_{Q} \frac{\left(\hat{u}_{1,m}(x)-\hat{u}_{1,m}(y)\right)\left(w_{k,m}(x)-w_{k,m}(y)\right)}{|x-y|^{n+2 s}} d x d y 
=\lambda_{1,k} \int_{\Omega} \hat{u}_{1,m} w_{k,m} d x .
    \end{equation}
Moreover, by the definition of $w_{k,m}$, we have
\begin{equation}\label{eq3.0.5t}
\int_{\Omega}\nabla \hat{u}_{1,m}\cdot \nabla w_{k,m} d x=\int_{\Omega \cap\left\{\hat{u}_{1,m}>d_{k}\right\}}\nabla \hat{u}_{1,m}\cdot \nabla v_{k,m} \,d x=\int_{\Omega}\left|\nabla w_{k,m}(x)\right|^2 d x .
    \end{equation}
 Hence, by non negativity of second integral in \eqref{eq3.0.4t} and from \eqref{eq3.0.5t}, we deduce the following
\begin{equation}\label{eqqwt}
\int_{\Omega}\left|\nabla w_{k,m}(x)\right|^2 d x \leq \lambda_1 \int_{\Omega} \hat{u}_{1,m} w_{k,m} d x, ~~(\text{since $\lambda_{1,k}\leq \lambda_1$}).
    \end{equation}
So,  using the Sobolev inequality,
 we have
\begin{equation}\label{eq3.0.6t}
\left(\int_{{  \Omega}}|w_{k,m}(x)|^{2^*}\,dx\right)^\frac{2}{2^*}\leq \mathcal{C} \int_{{  \Omega}}|\nabla w_{k,m}(x)|^2\,dx,~\text{for some}~ \mathcal{C}>0.
\end{equation}
 Suppose $x\in \left\{w_{k,m}>0\right\}$ then  we obtain
\begin{equation}\label{u_ot}
\hat{u}_{1,m}<2^k w_{k-1,m}, ~\forall~k\geq 1
    \end{equation}
    and also as a result (see Theorem 3.2 in \cite{franzina2013fractional}),
\begin{equation}\label{eq3.0.9t}
\left\{w_{k,m}>0\right\}=\left\{\hat{u}_{1,m}>d_k\right\} \subseteq\left\{\frac{1}{2^k}< w_{k-1,m}\right\}, ~\forall~k \geq 1.
\end{equation}
Now, from \eqref{eqqwt} and using  \eqref{u_ot}, $w_{k-1,m}\geq w_{k,m}$, we have
\begin{equation}\label{a.e.t}
\begin{aligned}
\int_{\Omega}\left|\nabla w_{k,m}(x)\right|^2 d x & \leq \lambda_1 \int_{\left\{w_{k,m}>0\right\}} \hat{u}_{1,m} w_{k,m} d x  \leq \lambda_1 2^k  \int_{\left\{w_{k,m}>0\right\}} w_{k-1,m} w_{k,m} d x \\
& \leq \lambda_1 2^k   \int_{{  \Omega}} w_{k-1,m}^2 d x= \lambda_1 2^k  \left\|w_{k-1,m}\right\|_{L^2({  \Omega})}^2  = \lambda_1 2^k   \mathcal{U}_{k-1,m} .
\end{aligned}
    \end{equation}
As a consequence, using \eqref{eq3.0.9t} in \eqref{a.e.t} and by the Chebyshev inequality, we obtain
\begin{equation}\label{new-tm1t}
\begin{aligned}
\mathcal{U}_{k-1,m} & =\int_{{  \Omega}} w_{k-1,m}^2 d x \geq \int_{\left\{w_{k-1,m}>\frac{1}{2^{k}}\right\}} w_{k-1,m}^2 d x  \geq \frac{1}{2^{2k}}\left|\left\{w_{k-1,m}>\frac{1}{2^k}\right\}\right| \geq \frac{1}{2^{2k}}\left|\left\{w_{k,m}>0\right\}\right| .
\end{aligned}
    \end{equation}
Now, we use the Hölder inequality, \eqref{eq3.0.6t}, \eqref{a.e.t} and \eqref{new-tm1t} to obtain the following estimate
\begin{equation}\label{eqht}
\begin{aligned}
\mathcal{U}_{k,m} =\left\|w_{k,m}\right\|_{L^2({  \Omega})}^2 \leq \left(\int_{{  \Omega}} |w_{k,m}|^{2^*} d x\right)^{\frac{2}{2^*}}\left|\left\{w_{k,m}>0\right\}\right|^{\frac{2}{n}}
&\leq \mathcal{C}  \left(\int_{{  \Omega}}\left|\nabla w_{k,m}\right|^2 dx\right) \left|\left\{w_{k,m}>0\right\}\right|^{\frac{2}{n}} \\
&\leq \mathcal{C}\left(\lambda_1 2^k  \mathcal{U}_{k-1,m}\right)\left(2^{2 k} \mathcal{U}_{k-1,m}\right)^{\frac{2}{n}}\\ 
&=\mathbf{c}^{\prime}\left(2^{1+\frac{4}{n}}\right)^{k-1} \mathcal{U}_{k-1,m}^{1+\frac{2}{n}},
\end{aligned}
    \end{equation}
where $\mathcal{C}>0$ is a Sobolev constant and $\mathbf{c}^{\prime}= \lambda_1 2^{1+\frac{4}{n}} \mathcal{C} $.
We can see  $\frac{2}{n}>0$,
$
r=2^{1+\frac{4}{n}}>1,
$
using the Lemma $7.1$ of \cite{MR1962933} that $\mathcal{U}_{k,m} \rightarrow 0$ as $k \rightarrow \infty$, provided that
$$
\mathcal{U}_0=\left\|\hat{u}_{1,m}\right\|_{L^2({  \Omega})}^2=\rho<\left(\mathbf{c}^{\prime}\right)^{\frac{-n}{2} } r^{\frac{-n^2}{4}} .
$$
As a consequence, if $\rho>0$ is small enough, we can use Dominated Convergence Theorem to conclude
$$
0=\lim _{k \rightarrow \infty} \mathcal{U}_{k,m}=\lim _{k \rightarrow \infty} \int_{{  \Omega}}\left(\hat{u}_{1,m}-d_k\right)_{+}^2 d x=\int_{{  \Omega}}\left(\hat{u}_{1,m}-1\right)_{+}^2 d x .
$$
Recalling $\hat{u}_{1,m}=\sqrt{\rho} u_{1,m}$ and $u_{1,m} \geq 0$, using above we obtain
$0 \leq u_{1,m} \leq \frac{1}{\sqrt{\rho}} \text{ a.e. in } {  \Omega}$,
which implies $u_{1,m} \in L^{\infty}({  \Omega})$.
{  Now,  if $x\in \mathcal{N}_m$, then using the definition of $\mathcal{N}_s$ (see  \eqref{normal}) and $\mathcal{N}_s u_{1,m}(x)=0$, we have
$$ u_{1,m}(x)\int_{\Omega} \frac{dy}{|x-y|^{n+2s}}= \int_{\Omega} \frac{u_{1,m}(y) dy}{|x-y|^{n+2s}},$$
which implies
$$
u_{1,m}(x)=\frac{\int_{\Omega} \frac{u_{1,m}(y) dy}{|x-y|^{n+2s}}}{\int_{\Omega} \frac{dy}{|x-y|^{n+2s}}}.$$
Since $u\in L^{\infty}(\Omega)$ then we get $|u_{1,m}(x)|\leq \|u_{1,m}\|_{L^{\infty}(\Omega)}$, for each $x\in\mathcal{N}_m$ and $m\in \N$. Thus, we conclude that $u_{1,m}\in L^{\infty}(\Omega_m).$
}

\end{proof}



The behaviour of the sequence $\{u_{1,k}\}$, when $k$ is large, is the subject of our next proposition.
\medskip
\begin{Proposition}\label{weakly}
Suppose that $\Omega$ is admissible domain and $\mathcal{N}_k$ and $\mathcal{D}_k$ are as in \eqref{nkdk} with $\Omega_k$ bounded for any $k$.  Consider the solutions $\{u_{1,k}\}$ of \eqref{sequence}, then there exists  $u^*$ in $\mathcal{X}^{1,2}_{\mathcal{D}}(\Omega_k)$  such that, up to a subsequence, as $k \to \infty$\\
\begin{align}\label{eq3.4}
\begin{cases}
u_{1,k} \rightharpoonup u^* \mbox{weakly in } \mathcal{X}^{1,2}_{\mathcal{D}}(\Omega_k)\\
u_{1,k} \to u^* \mbox{strongly in } \ {L^2_{\mathrm{loc}} (\R^n)},\\ 
u_{1,k} \to u^* \mbox{a.e in }  \R^n.
\end{cases}
\end{align}
\end{Proposition}
\begin{proof}
 Testing  \eqref{sequence} with $u_{1,k}$ itself and using Proposition \ref{u1k}, we find that
\begin{equation}\label{universal}
\int_{\Omega}|\nabla u_{1,k}|^2\,dx+ \int_{Q} \dfrac{(u_{1,k}(x)-u_{1,k}(y))^2}{|x-y|^{N+2s}}dxdy
=\lambda_{1,k}\dint_{\Omega}| u_{1,k}|^2\, dx\le \lambda_1(\R^n\setminus \Omega).
\end{equation}
As a consequence of Proposition \ref{u1k}, the  sequence $ \{\eta(u_{1,k})\}_{k\geq1}$  is uniformly bounded and thus, up to a subsequence, there exists $u^*\in  \mathcal{X}^{1,2}_{\mathcal{D}}(\Omega_k)$ such that
$$u_{1,k}\rightharpoonup u^*\hbox{   weakly in }  \mathcal{X}^{1,2}_{\mathcal{D}}(\Omega_k).$$
By the compact embedding $ \mathcal{X}^{1,2}_{\mathcal{D}}(\Omega_k)\hookrightarrow \hookrightarrow L^2_{\mathrm loc}(\R^n)$, we can infer that \eqref{eq3.4} holds true.
    
\end{proof}
Let us now describe an interesting property of functions satisfying the Neumann condition on a set whose measure approaches infinity. For the proof of the following lemma, we refer to [Lemma 3.3 in \cite{MR3784437}].
\begin{Lemma}
If 
$u$ satisfies the Neumann condition as
$$
\mathcal{N}_s u(x)=0,  \qquad \forall~ x\in \mathcal{N}\,,
$$
where $\mathcal{N}$ satisfies 
$$
 \mathcal{N}\cap B_R^c \neq \emptyset,\, \qquad\forall~ R >0, 
$$
where $B^c_{R}=\{x\in \R^n\setminus\Omega : ~|x|>R\}$.
Then,  for all sequences $\{x_j\}_{j}\subset \mathcal{N}$
 such that $ |x_j|\to \infty$  as $j \to + \infty$,
we have that $\{u(x_j)\}_j $ converges to its average on $\Omega$, that is
$$
\lim_{j \to \infty } u(x_j) = \frac{1}{|\Omega|}\int_{\Omega} u(x) dx.
$$
\end{Lemma}
\medskip 
Motivated by \cite{MR4151104} and \cite{MR1971262}  in the local context, our objective will now be to investigate what happens to the sequence $\{\lambda_{1,k}\}_{k\in\N}$ when the sets $\mathcal{D}_k$ and $\mathcal{N}_
k$ change with $k$. As we have already stated the fact that the boundary of the nonlocal framework is the entire $\R^n\setminus \Omega$ makes the situation different because the manner in which the sets can move or disappear may be much varied and complicated. 
 Before providing rigorous convergence results, we motivate readers to examine one example of these scenarios by referring to \cite{MR3784437}.
 
\vspace{0.5em}



\subsection{Dissipating Neumann sets}\label{secNeu}
\setcounter{equation}{0}

We consider $\{\mathcal{D}_{k}\}_{k\geq 1}$ and $\{\mathcal{N}_k\}_{k\geq 1}$ as sequences  of  open sets in $ \R^n\setminus \bar{\Omega}$ that  satisfy \eqref{nkdk} and $\Omega_k$ bounded. For each $k$, the pair $(\lambda_{1,k}, u_{1,k})$ denotes $L^2$ normalized solutions associated with \eqref{sequence}, precisely the first eigenvalue-positive normalized eigenfunction pair.
Now, we consider $\varphi_1\in\mathcal{X}^{1,2}_0(\Omega):= {\overline{C^{\infty}_0(\Omega)}}^{{{\eta(u)}}}$, the   first positive eigenfunction which solves the following Dirichlet problem
\begin{equation}\label{ev-diribdry}
\left\{\begin{array}{ll}
\mathcal{L} \varphi_1 &= \lambda_{1}\varphi_1\qquad \inn \Omega,\\
\varphi_1&= 0 ~~\inn  \R^n\setminus \Omega,
\end{array}\right.
\end{equation}
with $\|\varphi_1\|_{L^{2}(\Omega)}=1$.
We recall the following result at first.
\begin{Lemma}\label{integrable1}
Let $\Omega$ be admissible domain then $\psi\in L^{1}(\mathbb{R}^{n}\setminus\Omega)$, where for $x\in\R^n\setminus \Omega,$ 
$$
\psi(x)= \dint_{\Omega}\dfrac{ \varphi_1(y) }{|x-y|^{N+2s}}\,dy.
$$
\end{Lemma}
\begin{proof}
    We follow {Lemma 4.1} of \cite{MR3784437} for a proof. 
\end{proof}

\begin{Lemma}\label{lam}
       If $\varphi_1$ solves \eqref{ev-diribdry}, then
$\varphi_1\in C^{1,\alpha}(\bar\Omega)$, for some $ \alpha\in(0,1).$
 \end{Lemma}
\begin{proof}
The proof can be found in [Corollary 3.1 in  \cite{pezzo2019eigenvalues} or in Theorem 2.7 in \cite{Faber}].
\end{proof}

Now, we can establish the proof of our main result.\\

\textbf{Proof of  Theorem \ref{theoNeu}:}
  Taking $\varphi_1$   as  a test function in \eqref{sequence} and using Proposition \ref{P} with $v=u_{1,k}$ and $u=\varphi_1$ 
  we get
  \begin{equation} \label{eq4.2}
   \int_{\Omega} \nabla u_{1,k} \cdot\nabla \varphi_1\,dx+ \int_{Q} \dfrac{(u_{1,k}(x)-u_{1,k}(y))(\varphi_1(x)-\varphi_1(y))}{|x-y|^{N+2s}}dxdy
 =\lambda_{1,k}\dint_{\Omega} u_{1,k}(x) \varphi_1(x) \,dx
 \end{equation}
       and  
 \begin{equation}\label{eq4.3}
 \begin{split}
&\int_{\Omega} \nabla u_{1,k} \cdot\nabla \varphi_1\,dx+ \int_{Q} \dfrac{(u_{1,k}(x)-u_{1,k}(y))(\varphi_1(x)-\varphi_1(y))}{|x-y|^{N+2s}}dxdy\\
&\quad=\lambda_1\dint_{\Omega} u_{1,k}(x) \varphi_1(x) \,dx+\int_{\overline {{\mathcal{N}_k}}\cap \partial\Omega}u_{1,k}\frac{\partial \varphi_1 }{\partial \nu} \,d\sigma+\int_{{\mathcal{N}_k}}u_{1,k}(x) \mathcal{N}_s \varphi_1(x) \, dx.
\end{split}
       \end{equation}       
Using boundary conditions and subtracting  \eqref{eq4.3} from \eqref{eq4.2}),  we deduce that  
\begin{equation}\label{sub}
\begin{split}
(\lambda_{1}-\lambda_{1,k}) \dint_{\Omega}\varphi_1(x) u_{1,k}(x)\,dx=& -\int_{\overline {{\mathcal{N}_k}}\cap \partial\Omega}u_{1,k}\frac{\partial \varphi_1}{\partial \nu} \,d\sigma-\dint_{\mathcal{N}_k} u_{1,k}(x)\mathcal{N}_{s}\varphi_1\, dx\\
=&-\int_{\overline {{\mathcal{N}_k}}\cap \partial\Omega}u_{1,k}\frac{\partial \varphi_1}{\partial \nu} \,d\sigma+\dint_{\mathcal{N}_{k}} \dint_{\Omega} \dfrac{\varphi_{1}(y) u_{1,k}(x) }{|x-y|^{N+2s}}\ dy dx.
    \end{split}
\end{equation}
 Now passing $k\to \infty$ in \eqref{sub} and suppose, 
$
\lim_{k \to \infty} \lambda_1(\mathcal{D}_k)=  \lambda_1(\R^n\setminus \Omega) $
then this implies
\begin{equation}\label{eq4.6}
-\lim_{k\to \infty}\int_{\overline {{\mathcal{N}_k}}\cap \partial\Omega}u_{1,k}\frac{\partial \varphi_1}{\partial \nu} \,d\sigma+\lim\limits_{k\rightarrow\infty}\dint_{\mathcal{N}_{k}} \dint_{\Omega} \dfrac{\varphi_{1}(y) u_{1,k}(x) }{|x-y|^{n+2s}}\ dy dx=0.~
\end{equation}
Since $0\le \lambda_{1,k}\leq\lambda_{1}$ from Proposition \ref{u1k}(1), we obtain that
$$
0\le \liminf_{k\to\infty}  \lambda_{1,k} \le \limsup_{k\to\infty}  \lambda_{1,k} \le \lambda_1.
$$
{So our goal is to show that \eqref{eq4.6} is equivalent to $\displaystyle\lim\limits_{k\rightarrow \infty}\lambda_{1} (\mathcal{D}_k) =\lambda_{1} (\R^n\setminus \Omega)$}. 
To obtain this, we construct a subsequence $\{\lambda_{1,k_j}\}_{j\geq 1}$ {converging to the $\liminf\limits_{k\to\infty}\lambda_{1,k}$}  and from  Proposition \ref{weakly}, we get the following for the corresponding subsequence $\{u_{1,k_j}\}_{j\geq 1}$
\begin{align}
\begin{cases}
u_{1,k_j} \rightharpoonup u^* \mbox{ weakly in } \mathcal{X}^{1,2}_{{\mathcal{D}}}(\Omega_k)\\
u_{1,k_j} \to u^* \mbox{ strongly in } \ L^2 _{\mathrm{loc}}(\R^n),\\ 
u_{1,k_j} \to u^* \mbox{ a.e in }  \R^n,
\end{cases}
\end{align}
with $u^*\in \mathcal{X}^{1,2}_{\mathcal{D}}(\Omega_k)$  such that $u^*\gneq 0$, as mentioned in Proposition \ref{weakly}. Hence, $$\lim\limits_{j\rightarrow\infty} \dint_{\Omega} \varphi_{1}(x)u_{1,k_j}(x)\,dx= \dint_{\Omega} \varphi_{1}(x)u^{*}(x)\,dx >0. $$
If $\eqref{eq4.6}$ is true, then using above equality in (\ref{sub}), we can see that  $\lim\limits_{j\to\infty}  \lambda_{1,k_j} = \lambda_1$. Consequently, the initial subsequence also converges to $\lambda_1$. 
For complete the proof, take $R$ large enough so that $\Omega \subset B_{\frac{R}{2}}(0)$, then
$$
\begin{aligned}
I_k= \int_{\mathcal{N}_k} \int_{\Omega} \frac{\varphi_1(y) u_{1,k}(x)}{|x-y|^{n+2 s}} d y d x &\leq  \int_{\mathbb{R}^n \backslash B_R} \int_{\Omega} \frac{\varphi_1(y) u_{1,k}(x)}{|x-y|^{n+2 s}} d y d x+\int_{\mathcal{N}_k \cap B_R} \int_{\Omega} \frac{\varphi_1(y) u_{1,k}(x)}{|x-y|^{n+2 s}} d y d x \\
& = I_{1,k}+I_{2,k}.
\end{aligned}
$$
If $y \in \Omega$ and $|x|>R$, then $|x-y|>\frac{|x|}{2}$. Since $\varphi_1 \in C^{1,\alpha}(\bar{\Omega}),$ for some $ \alpha\in(0,1)$ from Lemma \ref{lam}, we have that
$$
I_{1,k} \leqslant C \int_{\mathbb{R}^n \backslash B_R} \int_{\Omega}\left(\frac{2}{|x| }\right)^{n+2 s} d y d x=C 2^{n+2 s}|\Omega| \int_{|x|>R} \frac{d x}{|x|^{n+2 s}}\leq C_{n, s}|\Omega| \frac{1}{R^{n+2 s}} .
$$
Hence for any $\varepsilon>0$, if we choose $R$  large enough then $I_{1,k} \leq \frac{\varepsilon}{2}, \forall~ k$.
Now using the Proposition \ref{u1k}(3) and Lemma \ref{integrable1}, we obtain
$$
I_{2,k} \leq C \int_{\mathcal{N}_k \cap B_R} \int_{\Omega} \frac{\psi_1(y)}{|x-y|^{n+2 s}} d y d x=C \int_{\mathcal{N}_k \cap B_R} \Phi(x) d x .
$$
Since the measure $d \mu=\Phi(x) d x$ is absolutely continuous with respect to the Lebesgue measure, there exists $\delta>0$, such that if $\mathcal{N}$ is a measurable subset of $\mathbb{R}^N \backslash \Omega$, with $|\mathcal{N}|<\delta$, then
$$
\int_{\mathcal{N}} \Phi(x) d x<\frac{\varepsilon}{2 C},
$$
for $\varepsilon$ given above. Hence, by assumptions of Theorem \ref{theoNeu}  implies the existence of $k_0>0$ such that, $\forall~ k \geqslant k_0$, $\left|\mathcal{N}_k \cap B_R\right|<\delta$. We conclude that $\left|I_{2,k}\right|<\frac{\varepsilon}{2}$ and therefore $I_k\leq \frac{\varepsilon}{2}+\frac{\varepsilon}{2}=\varepsilon, \forall~ k \geqslant k_0$.
Thus, \eqref{eq4.6} holds, then we conclude that 
$\displaystyle\lim\limits_{k\rightarrow \infty}\lambda_{1} (\mathcal{D}_k) =\lambda_{1} (\R^n\setminus \Omega)$.

\subsection{Dissipating Dirichlet sets}\label{secDir}
The main goal of this section is to repeat the analysis done in Section \ref{secNeu} when the Dirichlet sets dissipate and then to prove Theorem \ref{introd}.
\begin{Theorem}\label{N1}
If $\Omega$ is an admissible domain then 
$\lim\limits_{k\rightarrow \infty}\lambda_{1,k}=0$ if and only if 
$$\lim_{k\rightarrow\infty}\left( -\int_{\overline {{\mathcal{D}_k}}\cap \partial\Omega}u_{1,k}\frac{\partial \psi_1}{\partial \nu} \,d\sigma+\int_{\mathcal{D}_k}\int_\Omega\frac{\psi_1(x)u_{1,k}(y)}{|x-y|^{n+2s}}\,dydx\right)=0.$$
\end{Theorem}
\begin{proof}
Let $\psi_1$ denotes the first eigenfunction associated with $\lambda_1(\emptyset)$ i.e. it satisfies,
\begin{equation}\label{pneu}
\begin{cases}
\mathcal{L} \psi_1& = 0 \quad \inn \Omega,\\
\mathcal{N}_{s} \psi_1&= 0 ~~\text{in}~~ \R^n\setminus \bar\Omega\\
\frac{\partial \psi_1}{\partial\nu}&=0\qquad  \text{in}~~ \partial\Omega,
\end{cases}
\end{equation} 
and for more details, we refer \cite{mugnai2022mixed}.
Consequently,  we know that {$\psi_1= \dfrac{1}{|\Omega|^{1/2}}$}. Taking $u_{1,k}$ as  test function in \eqref{pneu} and  $\psi_1$ as test function in \eqref{sequence}, we have
\begin{equation}\label{4.2.2}
\begin{split}
\lambda_{1,k} \dint_{\Omega}\psi_1(x) u_{1,k}(x)\,dx=& -\int_{\overline {{\mathcal{D}_k}}\cap \partial\Omega}u_{1,k}\frac{\partial \psi_1}{\partial \nu} \,d\sigma-\dint_{\mathcal{D}_k} u_{1,k}(x)\mathcal{N}_{s}\psi_1\, dx\\
=&-\int_{\overline {{\mathcal{D}_k}}\cap \partial\Omega}u_{1,k}\frac{\partial \psi_1}{\partial \nu} \,d\sigma+\dint_{\mathcal{D}_{k}} \dint_{\Omega} \dfrac{\psi_{1}(x) u_{1,k}(y) }{|x-y|^{N+2s}}\ dy dx,
    \end{split}
\end{equation}
We deduce that
\begin{equation}\label{equivPsi}
\lim\limits_{k\rightarrow \infty}\lambda_{1,k}=0\quad \implies \quad  \lim_{k\rightarrow\infty}\left( -\int_{\overline {{\mathcal{D}_k}}\cap \partial\Omega}u_{1,k}\frac{\partial \psi_1}{\partial \nu} \,d\sigma+\int_{\mathcal{D}_k}\int_\Omega\frac{\psi_1(x)u_{1,k}(y)}{|x-y|^{n+2s}}\,dydx\right)=0,
\end{equation}
{The converse statement of \eqref{equivPsi} is also true by similar approach as section \eqref{secNeu} with the help of proof of Theorem \ref{theoNeu}. }
\end{proof}

\textbf{Proof of Theorem \ref{introd}:}
    
 By using the converse statement of \eqref{equivPsi} and assumption $\lim_{k\to \infty} |\overline {{\mathcal{D}_k}}\cap \partial\Omega|=0$, we have  
 \begin{equation}
  \lim_{k\rightarrow\infty}\int_{\mathcal{D}_k}\int_\Omega\frac{\psi_1(x)u_{1,k}(y)}{|x-y|^{n+2s}}\,dydx=0 \qquad \implies \qquad \lim\limits_{k\rightarrow \infty}\lambda_{1,k}=0.
\end{equation}
 
So,  our goal is to show  that
$$
\lim _{k \rightarrow \infty} \int_{\mathcal{D}_k} \int_{\Omega} \frac{\psi_1(x) u_{1,k}(y)}{|x-y|^{n+2 s}} d y d x=0,
$$
where $\psi_1=\frac{1}{|\Omega|^{1 / 2}}$ and solves Neumann problem \eqref{pneu}. Suppose  $\Omega \subset B_{R / 2}$, $R>0$ is large enough. We note that, from Proposition \ref{u1k}, and (using fact if $y\in \Omega$ and $|x|>R$ then $|x-y|>\frac{|x|}{2}$), there exists constant $C=C(n, \Omega, s)$ such that
\begin{align*}
J_k^1=\int_{\mathbb{R}^n \backslash B_R} \int_{\Omega} \frac{\psi_1(x) u_{1,k}(y)}{|x-y|^{n+2 s}} d y d x
&\leqslant C \int_{\mathbb{R}^n \backslash B_R} \int_{\Omega}\left(\frac{1}{|x| / 2}\right)^{n+2 s} d y d x = C 2^{n+2s}|\Omega|\int_{|x|>R} \frac{dx}{|x|^{n+2s}}\leqslant \frac{\tilde{C}}{R^{n+2 s}} .
\end{align*}
Thus, for given any $\varepsilon>0$, we may choose $R$ large enough such that
$$
J_k^1 \leqslant \frac{\varepsilon}{2} .
$$
Moreover, by (3) statement of the Proposition\ref{u1k}, we find 
\begin{equation}\label{eq5.4}J_k^2=\int_{\mathcal{D}_k \cap B_R} \int_{\Omega} \frac{\psi_1(x) u_{1,k}(y)}{|x-y|^{n+2 s}} d y d x \leqslant C \int_{\mathcal{D}_k \cap B_R} \int_{\Omega} \frac{1}{|x-y|^{n+2 s}} d y d x .
    \end{equation}
Since $0< s<\frac{1}{2}$. We utilize the [Lemma 2.3 and Lemma 2.6 in \cite{MR3784437}] and conclude that $J_k^2<\frac{\varepsilon}{2}$ for some $k$ large enough. Thus, for each $\varepsilon>0$, there exists $k_0>0$ such that
$$
\int_{\mathcal{D}_k} \int_{\Omega} \frac{\psi_1(x) u_{1,k}(y)}{|x-y|^{n+2 s}} d y d x \leqslant J_k^1+J_k^2=J_k<\varepsilon ~\forall~ k \geqslant k_0.
$$
Our goal is complete now.
\begin{remark}
    It should be noted that the restriction on $s$ occurs in order to estimate the term in \eqref{eq5.4}. In this case, we require $2s<1$ in order to apply the Integrability of [Lemma 2.6 in \cite{MR3784437}], because we can utilize the regularity of $\psi_1$, which is an eigenfunction of the Dirichlet problem, to reduce the singularity of the kernel, this restriction on $s$ does not apply to the case which we have been discussed in Section 4.1. We can provide a partial solution for the problem $\frac{1}{2}\leq s<1$. Particularly, when the Dirichlet sets do not collapse to the boundary of $\Omega$, we can demonstrate the following useful results.
\end{remark}
\begin{Proposition}\label{p5.4}
 Suppose $s \in(0,1)$ and
$
\lim _{k \rightarrow \infty}\left|\mathcal{D}_k \cap B_R\right|=0,\quad \forall~ R>0,
$
and
$$
\exists~ \delta, k_0>0 \text { such that } \operatorname{dist}\left(\mathcal{D}_k, \Omega\right)>\delta, \quad~ \forall~ k \geqslant k_0,
$$
then $\lim _{k \rightarrow \infty} \lambda_{1,k}=0$ up to a subsequence.
\end{Proposition}
\begin{proof}
It suffices to observe that, for sufficiently large $k$, equation \eqref{eq5.4} can be replaced by
$$
J_k^2 \leqslant C \delta^{-(n+2 s)}\left|\mathcal{D}_k \cap B_R\right|,
$$
since $|x-y| \geqslant \delta$ whenever $x \in \mathcal{D}_k \cap B_R$ and $y \in \Omega$. Thus, we can conclude its proof is the same as the Theorem \ref{introd}.
\end{proof}
To examine the case of Dirichlet sets that are arbitrarily close to $\Omega$, we define the following condition:
\begin{equation}\label{eq5.5}
\lim _{k \rightarrow \infty} \int_{\mathcal{D}_k} \int_{\Omega} \frac{1}{|x-y|^{n+2 s}} d y d x=0.
\end{equation}
\begin{Proposition}\label{p5.5}
 If $s \in(0,1)$ and $\Omega, \mathcal{D}_k, \mathcal{N}_k$ are defined as in equation \eqref{nkdk} and $\left\{\lambda_{1,k}\right\},\left\{u_{1,k}\right\}$ the corresponding the same sequences of eigenvalues and eigenfunctions as previously. Then, if $\eqref{eq5.5}$ holds for $s$, then
$$
\lim _{k \rightarrow \infty} \lambda_{1,k}=0
$$
\end{Proposition}
\begin{proof}

Since $\lambda_{1,k's}\geq 0$ and bounded. Our goal is to prove 
$
\limsup _{k \rightarrow \infty} \lambda_{1,k}=0 \text {. }$
Take a subsequence $\left\{\lambda_{1,k_j}\right\}_j$ that converging to the $\lambda '=\lim \sup _{k \rightarrow \infty} \lambda_{1,k}$. We can suppose that the associated sequence of eigenfunctions $\left\{u_{1,k_j}\right\}_j$ converges weakly in $\mathcal{X}^{1,2}_{\mathcal{D}}(\Omega_k) $ to the function $u^*$ obtained in Proposition \ref{weakly}, by taking a sub subsequence if needed.
 Let  $\phi$ be a bounded test function then we have
\begin{equation}\label{eq5.6}
\begin{split}
\int_{\Omega} \nabla(u_{1,k_j}(x))\cdot\nabla\phi(x)\,dx+ \int_{Q} \frac{\left(u_{1,k_j}(x)-u_{1,k_j}(y)\right)(\phi(x)-\phi(y))}{|x-y|^{n+2 s}} d x d y
&=\lambda_{1,{k_j}} \int_{\Omega} \phi u_{1,k_j} d x\\
&\quad-\int_{\mathcal{D}_{k_j}} \phi \mathcal{N}_s (u_{1,k_j}) d x .
\end{split}
\end{equation}
By using equation \eqref{eq5.5}, we obtain
$$
\begin{aligned}
\lim _{j \rightarrow \infty}\left|\int_{\mathcal{D}_{k_j}} \phi \mathcal{N}_s( u_{1,k_j} )d x\right| & \leqslant \lim _{j \rightarrow \infty} \int_{\mathcal{D}_{k_j}} \int_{\Omega} \frac{|\phi(x)| |u_{1,k_j}(y)|}{|x-y|^{n+2 s} } d y d x \leqslant C \lim _{j \rightarrow \infty} \int_{D_{k_j}} \int_{\Omega} \frac{1}{|x-y|^{n+2 s}} d y d x=0 .
\end{aligned}
$$
So, taking limits on both sides  in \eqref{eq5.6}, as $j \to \infty$, we have
$$
\int_{\Omega} \nabla(u^*(x))\cdot\nabla\phi(x)\,dx+ \int_{Q} \frac{\left(u^*(x)-u^*(y)\right)(\phi(x)-\phi(y))}{|x-y|^{n+2 s}} d x d y=\lambda ' \int_{\Omega} \phi u^* d x.
$$
Since we also have
$$
\mathcal{N}_s u^*(x)=0, \quad \text { a.e. on } \R^n\setminus\bar{\Omega},
$$
we find that $u^*$ is a solution to the problem
\begin{equation}
\begin{cases}
\mathcal{L} v ~~= \lambda' v \quad &\inn \Omega,\\
\mathcal{N}_{s} v = 0 &\text{in}~ \R^n\setminus\bar\Omega\\
\frac{\partial v}{\partial\nu}~~=0\qquad & \text{in}~~ \partial\Omega.
\end{cases}
\end{equation} 
Hence, either $u^* \equiv 0$ that is a contradiction with the fact that $\|u^*\|_{L^2(\Omega)}=1$, or $\lambda '=0$ which is the limit of
$\lim \sup _{k \rightarrow \infty} \lambda_{1,k}$ i.e.
$$\lim \sup _{k \rightarrow \infty} \lambda_{1,k}=\lambda '=0.$$
    
\end{proof}
\begin{remark}
We can see that equation $\eqref{eq5.5}$ is significantly stronger than the condition of the Theorem \ref{theoNeu}, due to assuming $\Omega \subset B_R$, we have
$$
\int_{\mathcal{D}_k} \int_{\Omega} \frac{1}{|x-y|^{n+2 s}} d y d x \geqslant \int_{\mathcal{D}_k \cap B_R} \int_{\Omega} \frac{1}{|x-y|^{n+2 s}} d y d x \geqslant \frac{|\Omega|}{(2 R)^{n+2 s}}\left|\mathcal{D}_k \cap B_R\right| .
$$
\end{remark}
\section{ Application to bifurcation results}
In this section, we shall study the bifurcation properties of the  following problem
 \begin{equation*}\label{ql} 
\left\{\begin{split} \mathcal{L}u\: &= \lambda h(u),~~u>0~ \text{in} ~\Omega, \\
      u&=0~~\text{in} ~~{U^c},\\
 \mathcal{N}_s(u)&=0 ~~\text{in} ~~{\mathcal{N}}, \\
 \frac{\partial u}{\partial \nu}&=0 ~~\text{in}~~ \partial \Omega \cap \overline{\mathcal{N}},
    \end{split} \right.\tag{$Q_\lambda$}
         \end{equation*}
where $\lambda>0$ and  $h$ is an asymptotically linear function that satisfies $(f1), \; (f2)$ and $(f3)$ conditions (see Section 1).
{We extend the continuous function $h$ to whole $\mathbb{R}$ in such a way that 
 $ h(t)= 0$ for $t\leq 0.$ The symbol used to represent this extension will remain unchanged.}

Consider the following Banach space
$$X= \{u\in C(\mathbb R^n):~u \equiv 0~ a.e. ~\text{in}~ {U^c} \}$$
equipped with the norm
$\|u\|=\sup_{x\in \mathbb R^n}|u(x)|<+\infty$, which satisfies the continuous embedding $X\hookrightarrow L^2(U)$. So, we consider $L_0$ to be the linear operator induced by $\mathcal{L}$ in $L^2(U)$ such that its domain is
\[D(L_0)= \mathcal{X}^{1,2}_{\mathcal{D}}(U)\cap H^2(U).\]
Next, we consider $L$ to be a restriction of $L_0$ on $X$ and its domain is defined as
\[D(L)= \{u\in X: ~ u\in D(L_0), ~L_0u\in X\}.\]
By $L(u)\geq 0$ in $\Omega$, we mean that
\begin{equation}\label{dd5.1}
    \int_{\Omega} \nabla u.\nabla \varphi \,dx +\int_{Q} \frac{(u(x)-u(y))(\varphi(x)-\varphi(y))}{|x-y|^{n+2s}} dxdy \geq 0,
\end{equation}
for every non negative $\varphi \in X$.
With this and the continuous extension of $h$ over whole $\mathbb R^n$, we establish the following weak maximum principle.
\begin{Lemma}\label{weak-max}
    Let $u\in D(L)$ such that $L(u)\geq 0$ in $\Omega$ then $u\geq 0$ in $U$.
\end{Lemma}
\begin{proof}
    On contrary, suppose $u<0$ in some non zero-measure subset of $U$, then the support of $u_{-}=\max\{-u,0\}\geq 0$ that {belongs to $X$} has non zero measure. So taking $\varphi =u_{-}$ in \eqref{dd5.1}, we get 
    \[\int_{\Omega} \nabla u.\nabla u_{-} \,dx +\int_{Q} \frac{(u(x)-u(y))(u_{-}(x)-u_{-}(y))}{|x-y|^{n+2s}} dxdy \geq 0,\]
    where $\int_{\Omega} \nabla u.\nabla u_{-} \,dx=-\int_{\Omega}|\nabla u_{-}|^2~dx\leq 0$ and setting $u_{+}=\max\{u,0\}$ we have
    \begin{align*}
        \int_{Q} \frac{(u(x)-u(y))(u_{-}(x)-u_{-}(y))}{|x-y|^{n+2s}} dxdy & = \int_{Q} \frac{(u_{+}(x)-u_{+}(y))(u_{-}(x)-u_{-}(y))}{|x-y|^{n+2s}} dxdy-\int_{Q} \frac{(u_{-}(x)-u_{-}(y))^2}{|x-y|^{n+2s}} dxdy \\
        & < \int_{Q} \frac{(u_{+}(x)-u_{+}(y))(u_{-}(x)-u_{-}(y))}{|x-y|^{n+2s}} dxdy\leq 0,
    \end{align*}
    since for each $x,y\in \mathbb R^n$, one can easily verify that
    \[(u_{+}(x)-u_{+}(y))(u_{-}(x)-u_{-}(y)) \leq 0. \]
This is a contradiction and hence we conclude that $u\geq 0$ in $U$.
\end{proof}
 
\begin{remark}\label{uniques}
  We can follow the proof of  Theorem 1.1 in \cite{biagi2022mixed}, to say that
 for any nonnegative  and nontrivial $w \in L^2(U)$,  problem
  \begin{equation}\label{w} 
\left\{\begin{split} Lu\: &= w,~~u>0~ \text{in} ~\Omega, \\
      u&=0~~\text{in} ~~{U^c},\\
 \mathcal{N}_s(u)&=0 ~~\text{in} ~~{\mathcal{N}}, \\
 \frac{\partial u}{\partial \nu}&=0 ~~\text{in}~~ \partial \Omega \cap \overline{\mathcal{N}},
    \end{split} \right.
         \end{equation}
has a unique weak solution $u\in \mathcal{X}^{1,2}_\mathcal{D}{(U)}$ satisfying $\eta(u)\leq c \|w\|_{L^2(U)}$,
 where  $c> 0$ is an independent constant
of $w$.  
\end{remark}
For the compactness property of the operator $L^{-1}=K$, we are following the main ideas of the proof of Lemma 4.12 in \cite{Bal}. 
\begin{Lemma}\label{cpto}
    The operator ${L}^{-1}=K: X \rightarrow X$ is compact.
\end{Lemma} 
\begin{proof}
We shall prove this lemma in the following steps.

 Step 1: Consider \eqref{w} and first we show   
  $$\|K(w)\|_{L^{\infty}(\Omega)}=\|u\|_{L^{\infty}(\Omega)} \leq c\|w\|_{L^{\infty}(\Omega)}, ~\text{for any}~ w \in L^{\infty}(\Omega).$$
  We define $\mathcal{A}(k)=\{x \in \Omega: ~|u(x)| \geq k\}$, for any $k>0$.
  Choosing
  \begin{equation}
\varphi_k(x)=(\operatorname{sgn} u)\max {(|u|-k,0)}=
\left\{\begin{matrix} 
{   u-k, ~~~\text{if}~ u\geq k},\\[2mm]  
0, ~~~~~~~~~~~~\text{if}~~  |u|\leq k,\\[2mm]
u+k, ~~~~~~\text{if}~ ~u\leq -k,
\end{matrix}\right. \end{equation}
  as a test function in \eqref{w}, we have
\begin{equation}\label{bieq}
\int_{\Omega} \nabla u \cdot \nabla \varphi_k d x+\int_{Q}  \frac{(u(x)-u(y))\left(\varphi_k(x)-\varphi_k(y)\right)}{|x-y|^{n+2 s}} d x d y=\lambda\int_{\Omega} w \varphi_k d x.
    \end{equation}
Hence, by non-negativity of the second integral in \eqref{bieq} 
 and  using the Sobolev embedding and H\'older inequality, we obtain
$$
\begin{aligned}
\int_{\Omega}\left|\nabla \varphi_k\right|^2 d x=\int_{\Omega} \nabla u \cdot \nabla \varphi_k d x & \leq \lambda\int_{\Omega} w \varphi_k d x \leq\lambda\|w\|_{L^{\infty}(\Omega )} \int_{\mathcal{A}(k)} |\varphi_k| d x \\
& \leq\lambda C_0\|w\|_{L^{\infty}(\Omega)}|\mathcal{A}(k)|^{\frac{r-1}{r}}\left(\int_{\Omega}\left|\nabla \varphi_k\right|^2 d x\right)^{\frac{1}{2}}
\end{aligned}
$$
where $C_0$ is the Sobolev constant and $2\leq r\leq 2^*$. Hence, we have
\begin{equation}\label{515eq}
    \int_{\Omega}\left|\nabla \varphi_k\right|^2 d x \leq C_1{\lambda^2}\|w\|_{L^{\infty}(\Omega)}^2|A(k)|^{\frac{2(r-1)}{r}}.
\end{equation}
It is easy to check that if $h>k$ then $\mathcal{A}(h) \subset \mathcal{A}(k)$. Using this fact and \eqref{515eq}, we find
$$
\begin{aligned}
(h-k)^2|\mathcal{A}(h)|^{\frac{2}{r}} &\leq\left(\int_{\mathcal{A}(h)}(|u(x)|-k)^r d x\right)^{\frac{2}{r}}  \leq\left(\int_{\mathcal{A}(k)}(|u(x)|-k)^r d x\right)^{\frac{2}{r}} \\
& \leq C_2\int_{\Omega}\left|\nabla \varphi_k\right|^2 d x  \leq C{\lambda^{2}}\|w\|_{L^{\infty}(\Omega)}^2|\mathcal{A}(k)|^{\frac{2(r-1)}{r}}, \text{where}~ C_1C_2= C.
\end{aligned}
$$
Therefore, we have $|\mathcal{A}(h)| \leq C{\lambda^{r}} \frac{\|w\|_{L^{\infty}(\Omega)}^r}{(h-k)^r}|\mathcal{A}(k)|^{r-1}$, $\forall$ $h>k>0$. Thus using [Lemma 14 in \cite{MR3393266}] or [Lemma B.1 in \cite{Stampacchia}] we obtain $|\mathcal{A}(d)|=0$, where $d^r=c \lambda^r\|w\|_{L^{\infty}(\Omega)}^r 2^{\frac{r^2}{r-1}}$, and $0<c=c(r,\Omega, C_1,C_2)$. 

Hence, \begin{equation}\label{stp1}
\|u\|_{L^{\infty}(\Omega)} \leq c\|w\|_{L^{\infty}(\Omega)}.\
    \end{equation}

Step 2: Suppose $\left\{w_k\right\}$ be a bounded sequence in $X$. {From the proof of Theorem \ref{thm5.2} (see Appendix)} and using [Theorem 1.4 in \cite{su2022regularity}], we have for each $u_k=K(w_k)\in X$
\begin{equation}\label{w2p}
{\left\|u_k\right\|_{W^{2, p}(\Omega)} \leq C(n, s, p)\left(\left\|u_k\right\|_{L^p(\Omega)}+\left\|w_k\right\|_{L^p(\Omega)}\right)  \leq C'\left\|w_k\right\|_{L^{\infty}(\Omega)} \leq C'},
    \end{equation}
{  where we used \eqref{stp1}. Thus $\left\{u_k\right\}$ is a bounded sequence in ${W^{2, p}(\Omega)}$. From the compact embedding of $W^{2,p}(\Omega)$ in $C^{1,\beta}(\bar\Omega)$ for any $p>n$, we get that $\left\{u_k\right\}$ has a convergent subsequence in $C^{1,\beta}(\bar\Omega)$. Hence $\{u_k\}$ has a convergent subsequence in $C^{0,\beta}(\mathbb R^n)$. Now using the compact embedding of $C^{0,\beta}(\mathbb{R}^n)\cap X$ in $X$, we conclude the proof.}
    \end{proof}

Let us recall that the following eigenvalue problem 
\begin{equation} \label{3eq}
    \left\{\begin{split} L \varphi\: &= \lambda \varphi, \quad \varphi>0 ~~ \text{in} ~\Omega, \\
      \varphi&=0~~\text{in} ~~U^c,\\
 \mathcal{N}_s(\varphi)&=0 ~~\text{in} ~~{{\mathcal{N}}}, \\
 \frac{\partial \varphi}{\partial \nu}&=0 ~~\text{in}~~ \partial \Omega \cap \overline{{\mathcal{N}}}.
    \end{split} \right.
\end{equation}
has a principle eigenvalue $\lambda_1(\mathcal{D})>0$ with associated eigenfunction $0<\varphi_1$ belonging to $\mathcal{X}^{1,2}_\mathcal{D}{(U)}\cap C^{0,\beta}(\mathbb R^n)$, by Proposition \eqref{prop-1} and Lemma \ref{reg-1}. 
Let us consider a map $I: [0,\infty)\times X \to X$ by
\begin{equation}
I_{\lambda}(u)=I(\lambda, u)=u-\Psi_\lambda(u),   
\end{equation}
where
 $\Psi_\lambda(u)=K\left[\lambda \theta u+\lambda f( u)\right]=\lambda K[h(u)]$ and $Kh: X\to X$ is a compact operator, since $h: X\to X$ is a continuous function and ${L}^{-1}=K : X\to X$ is  compact due to Lemma \ref{cpto}.
 
We define the set $\Gamma$ as closure of the solution set $\{(\lambda,u)\in \mathbb{R}^+\times (X\setminus\{0\}):~I(\lambda,u)=0\}$ i.e. it is closure of the set of all nontrivial solutions of $I(\lambda, u)=0$, for each $\lambda>0$. Since $I(\lambda,0)=0,~\forall~ \lambda\in\mathbb{R}$, $u=0$ will be a trivial solution. In order to study the behaviour of $\Gamma$, it is convenient to define the following.

\begin{definition}\label{rem5.1}
We define for fixed $\lambda\in \mathbb{R}^+$, the Leray Schauder index (denoted by 'ind') as ind$(I_\lambda, 0)= \lim\limits_{\epsilon\to 0} deg(I_{\lambda}, B_{\epsilon}(0), 0)$, where $ B_{\varepsilon}(0)=\{u\in X: \|u\|<\varepsilon\}$. For details on degree theory, we refer to \cite{Ambrosetti}.
        \end{definition}
The following lemma is an important observation about the range of $\Gamma$ with respect to $\lambda$.
        \begin{Lemma}\label{5.1lem}
 Suppose \eqref{ql} has a solution then $\lambda< \theta\lambda_1(\mathcal{D})$.
    \end{Lemma} 
\begin{proof}
 Let $u>0$ in $\Omega$ be a solution to problem \eqref{ql}. Taking $\varphi_1$ as a test function in \eqref{ql} and using $f(t)>0,~ \forall~ t\in \mathbb{R}$, it follows that
$$
\lambda_1(\mathcal{D}) \int_{\Omega} u \varphi_1 d x=\int_{\Omega} \mathcal{L} u \varphi_1 d x=\lambda \int_{\Omega}  h(u) \varphi_1 d x=\lambda \int_{\Omega} (\theta u+f(u))\varphi_1\,dx.
$$
This implies
$$
\lambda_1(\mathcal{D}) \int_{\Omega} u \varphi_1 d x>\lambda \theta\int_{\Omega} u \varphi_1 dx 
$$
that is, $\lambda_1(\mathcal{D})>\theta\lambda$, which ends the proof.
    \end{proof}
    



\subsection{Bifurcation from zero}
 For the study of bifurcation from the line of trivial solutions, let us suppose, throughout this section, that conditions $(f_1)$, $(f_2)$, $(f_3)$ and $h(0)=0$ hold.
We begin with the following two lemmas which are crucial parts for establishing that $\lambda_0$(defined in \eqref{lem0}) is a bifurcation point from the line of trivial solutions.

\begin{Lemma}\label{lem5.3}
\begin{enumerate}[(a)]
Suppose $0<\lambda<\lambda_0.$ Then
\item  $\exists~\delta>0$ such that
$
I_{t \lambda}(u) \neq 0, ~\forall~u\in X,~ t \in[0,1],~  \text{whenever}~  0<\|u\|\leq\delta.
$
\item $\lambda_{0}$ is the only possible bifurcation point from line of trivial solutions for \eqref{ql}.
\item  $\operatorname{ind}\left(I_\lambda, 0\right)=1$, for all $\lambda<\lambda_0$.
\end{enumerate}
  \end{Lemma}
\begin{proof} $(a)$ 
 Supposing the contrary, there exists a sequence $\left\{u_k\right\} \subset X \backslash\{0\}$, $t_k \in[0,1]$ such that
$$
u_k \rightarrow 0 \text { in } X \quad \text { and } \quad I_{t_k \lambda}\left(u_k\right)=0,
$$
that is
$$ u_k- t_k\lambda K\left[h\left(u_k\right)\right]=0 \text{ or }
u_k=a_k K\left[h\left(u_k\right)\right] \text {, where } a_k=t_k \lambda .
$$
By Lemma \ref{weak-max}, $(f1)$ and $(f1)_0$, we get $u_k\geq 0$ in $U$. Setting
$$
w_k=\frac{u_k}{\left\|u_k\right\|}, ~\text{where} \left\|u_k\right\| \neq 0,
$$
we realise that
\begin{equation}\label{a_k}
w_k=a_k K\left[\frac{h\left(u_k\right)}{\left\|u_k\right\|}\right].
\end{equation}
From condition $(f2)$, we have 
$$
\left|\frac{h\left(u_k\right)}{\left\|u_k\right\|}\right| = \frac{\left|\theta u_k+ f(u_k)\right|}{\left\|u_k\right\|} \leq C'\left\|w_k\right\|=C', ~\text{where}~C'=C'(\theta,C)~ \text{and}~ C,C'\geq 0.
$$
Thus, the sequence $\left\{\frac{h\left(u_k\right)}{\left\|u_k\right\|}\right\} \subset X$ is bounded. Since $K$ is compact, there are $\left\{u_{k_j}\right\} \subset\left\{u_k\right\}$ and $v \in X$ such that
$$
K\left(\frac{h\left(u_{k_j}\right)}{\left\|u_{k_j}\right\|}\right) \rightarrow v \text { in } X.
$$
Consequently, there is a non negative $w \in X \backslash\{0\}$ such that
$$
w_{k_j} \rightarrow w \text { in } X.
$$
Then due to $(f3)$, we get 
$$
\frac{h\left(u_{k_j}\right)}{\left\|u_{k_j}\right\|}=\frac{h\left(u_{k_j}\right)}{u_{k_j}} \frac{u_{k_j}}{\left\|u_{k_j}\right\|}=\frac{h\left(u_{k_j}\right)}{u_{k_j}} w_{k_j}  \to aw \text { in } X.
$$
Suppose that
$$
t_{k_j} \rightarrow t_0 \quad \text { and } \quad t_{k_j}\lambda\rightarrow t_0 \lambda=\bar{\lambda}(say).
$$
Now, if we pass the limit $k_j \rightarrow+\infty$ in \eqref{a_k} the we get
$$
w=\bar{\lambda}a K( w),\text { where } \|w\|=1 \text { and }~ w \geq 0~   \text{in}~ U.
$$
Therefore, $w$ is an eigenfunction of $L$ associated with the eigenvalue $\lambda^*=\bar{\lambda}a$, that is
$$
{L w=\lambda^*  w ~\text { in }~ \Omega \text {. }}
$$
From  Lemma \ref{Ns} and Lemma \ref{strmx}, it follows that $w>0$ in $U$.
So, $w>0$ is the eigenfunction associated with the eigenvalue $\lambda^*$ of the operator $L.$ But the only eigenvalue with a positive eigenfunction is $\lambda_1(\mathcal{D})$, according to Proposition \ref{prop-1}. Hence, eigenfunction $w=\varphi_1$ and $\lambda^*=\bar{\lambda}a=\lambda_1(\mathcal{D}),$ i.e. $\bar{\lambda}=\lambda_0$(recall \eqref{lem0}) which is a contradiction, since
$
\bar{\lambda}=t_0 \lambda \leq \lambda<\lambda_0.  
$
Statement $(b)$ follows immediately from $(a).$

To prove statement $(c)$, 
we consider the homotopy $H(t,u)= I_{t\lambda}(u)$. Using  the homotopy invariance of the Leray-Schauder degree, $\forall ~\varepsilon \in\left(0, \delta\right]$, we get
$$
\operatorname{deg}\left(H(1,.), B_{\varepsilon}(0), 0\right)=\operatorname{deg}\left(H(0,.), B_{\varepsilon}(0), 0\right)=1,~ \text{namely}
$$
$$
\operatorname{deg}\left(I_\lambda, B_{\varepsilon}(0), 0\right)=\operatorname{deg}\left(I, B_{\varepsilon}(0), 0\right)=1, ~\text{where}~ I= \text{identity operator}.
$$
Hence, by the definition of index (see Remark \eqref{rem5.1}), we obtain $\operatorname{ind}\left(\Psi_\lambda, 0\right)=1.$
\end{proof} 
\begin{Lemma}\label{lem5.4}
Suppose $\lambda>\lambda_0$. Then
\begin{enumerate}[(a)]
\item  $\exists~ \delta>0$ such that  $I_\lambda(u) \neq b \varphi_1,~\forall ~ b \geq 0,
~ u \in X\text { satisfying } 0<\|u\| \leq \delta$. 
    \item $\operatorname{ind}\left(I_\lambda, 0\right)=0$ for all $\lambda>\lambda_0$. 
\end{enumerate}
     \end{Lemma}
\begin{proof}
Recall that $\varphi_1$ solves $L\varphi_1 =  \lambda_1(\mathcal{D}) \varphi_1, ~\varphi_1>0 ~
 \text{in }\Omega,$
that is, equivalently
\begin{equation}\label{eqt}
\varphi_1=\lambda_1(\mathcal{D}) K \varphi_1.
    \end{equation}
Let us first prove $(a)$. We suppose to the contrary that  there exists a sequence $\left\{u_k\right\} \subset X \backslash\{0\}$ and $\left\{b_k\right\} \subset\mathbb{R}^+$ with
$$
u_k \rightarrow 0 \text { in } X \quad \text { and } \quad I_\lambda(u_k)=b_k \varphi_1, \text { for all } k \in \mathbb{N}.
$$
From  \eqref{eqt}, we obtain $u_k=\lambda K[h(u_k)]+b_k\lambda_1(\mathcal{D})K\varphi_1$  which implies
 \[ Lu_k =  \lambda  h\left(u_k\right)+\lambda_1(\mathcal{D}) b_k  \varphi_1\]
and {$u_k\geq 0~
 \text{in }\Omega$, by Lemma \ref{weak-max}}. Taking $\varphi_1$ as a test function above and integrating both sides over $\Omega$, we get
\begin{equation}\label{505}
\int_{\Omega}(L u_k) \varphi_1 d x=\lambda \int_{\Omega}  h\left(u_k\right) \varphi_1 d x+\lambda_1(\mathcal{D}) b_k \int_{\Omega}  \varphi_1^2 d x.
    \end{equation}
Imposing $L\varphi_1 =  \lambda_1(\mathcal{D}) \varphi_1 ~
 \text{in }\Omega,$ on the left-hand side of above, we obtain
\begin{equation}\label{507}
\lambda_1(\mathcal{D}) \int_{\Omega}  \varphi_1 u_k d x=\lambda \int_{\Omega}  h\left(u_k\right) \varphi_1 d x+\lambda_1(\mathcal{D}) b_k \int_{\Omega} \varphi_1^2 d x \geq \lambda \int_{\Omega}  h\left(u_k\right) \varphi_1 d x.
    \end{equation}
{Due to $(f3)$, there exists a $\varepsilon_0>0$ such that
$$
h(t)>\frac{\lambda_0}{\lambda}a t, 
\text { for all } t \in(0, \varepsilon_0),
$$}
 therefore, for $k$ large enough, we can write 
\begin{equation}\label{508}
h\left(u_k\right)>\frac{\lambda_0}{\lambda}a u_k.
    \end{equation}
Using \eqref{507} and \eqref{508}, we find that 
$$
\lambda_1(\mathcal{D}) \int_{\Omega}  \varphi_1 u_k d x>\lambda \int_{\Omega}  \frac{\lambda_0}{\lambda}a u_k \varphi_1 d x=\lambda_0a \int_{\Omega}  u_k \varphi_1 d x.
$$
Hence
$
\lambda_1(\mathcal{D})>\lambda_0a=\lambda_1(\mathcal{D})
$
which is a contradiction, establishing the claim.

To prove statement $(b)$, 
 let us suppose $0<\varepsilon \leq \delta$, where $\delta>0$ is same as in statement $(a)$. As $I_\lambda\left(\overline{B_{\varepsilon}(0)}\right) \subset X$ is a bounded set, due to statement $(a)$, there exists a $b>0$ which is large enough such that
$$
I_\lambda(u) \neq b \varphi_1, \quad \forall~ u \in \overline{B_{\varepsilon}(0)}.
$$
By $(a)$, we also have
$$
I_\lambda(u) \neq t b \varphi_1, \quad \text { for } 0<\|u\| \leq \varepsilon \text { and } t \in[0,1].
$$
Using the homotopy $H(t,v)=I_{\lambda}(u)-tb\varphi_1$ \text{on the ball}~$B_{\varepsilon}(0)$, we get
$$
\operatorname{deg}\left(H(1,.), B_{\varepsilon}(0), 0\right)=\operatorname{deg}\left(H(0,.), B_{\varepsilon}(0), 0\right)=0,~ \text{namely}
$$
$$
\operatorname{deg}\left(I_\lambda, B_{\varepsilon}(0), 0\right)=\operatorname{deg}\left(I_\lambda-b\varphi_1, B_{\varepsilon}(0), 0\right)=0.
$$
Hence, by the definition of index (see Remark \eqref{rem5.1}), we get $\operatorname{ind}\left(I_\lambda, 0\right)=0$. This finishes the proof.

It is now possible to establish the proof of Theorem \ref{thm210} presented below.\\

 \textbf{Proof of Theorem \ref{thm210}:} The proof of the global bifurcation Theorem given by Rabinowitz \cite{Rabinowitz} and in the specified form [\cite{MR1956130}, Theorem 9.1.1], can be repeated with the help of Lemmas \ref{lem5.3} and \ref{lem5.4}, which ensure the existence of $\Gamma_0$.
 To begin, we need to demonstrate that $\left(\lambda_0, 0\right)$  is a bifurcation point of $I_{\lambda}(u)=0$ i.e. $u=\lambda K[h(u)]$ in $X$. 
Otherwise,
$$
I_\lambda(u) \neq 0, \forall~ \lambda \in\left[\lambda_0-\varepsilon, \lambda_0+\varepsilon\right] \text { and } 0<\|u\|\leq\varepsilon, \text{for some}~\varepsilon>0.
$$
 Thus, there exist $\mu_1$ and $\mu_2$ such that
$$
\lambda_0-\varepsilon<\mu_1<\lambda_0<\mu_2<\lambda_0+\varepsilon
$$
and
$$
\operatorname{deg}\left(I_{\mu_1}, B_{\varepsilon}(0), 0\right)=\operatorname{deg}\left(I_{\mu_2}, B_{\varepsilon}(0), 0\right),
$$
therefore,
$$
\operatorname{ind}\left(I_{\mu_1}, 0\right)=\operatorname{ind}\left(I_{\mu_2}, 0\right)
$$
 that contradicts the Lemmas \ref{lem5.3} and \ref{lem5.4}. Hence, $\lambda_0$ is a bifurcation point of $u=\lambda K[h(u)]$, and the existence of $\Gamma_0$ is showed. Furthermore, the reasoning employed in Lemma \ref{lem5.3} guarantee that $\lambda_0$ is the only bifurcation point for $u=\lambda K[h(u)]$.
 
Now we shall show that $\Gamma_0$ is unbounded. We define the solution operator to \eqref{ql} as $I:  [0,\infty)\times X  \longrightarrow X $ defined by
$$
\begin{aligned}
(\lambda, u)  ~\longmapsto u-\lambda K(h(u)).
\end{aligned}
$$
Then,  one can show that $I$ is analytic, see [\cite{MR3311916}, Proposition 2.3]. For all $(\lambda, u) \in [0,\infty) \times X$, the linearised operator $\partial_u I(\lambda, u): X \longrightarrow X$ is defined as
$$
\begin{aligned}
\varphi  \longmapsto ~\varphi-\lambda K\left(h^{\prime}(u) \varphi\right),
\end{aligned}
$$
where $\partial_u I(\lambda, u)=I d+\partial_u G(\lambda,u)$ and $\partial_u G(\lambda,u): X  \longrightarrow X$ is defined as
$$
\begin{aligned}
\varphi  \longmapsto-\lambda K\left(h^{\prime}(u) \varphi\right)
\end{aligned}
$$
which is compact using $K$ is compact, see Lemma \ref{cpto}. So, using [\cite{MR1956130}, Theorem 2.7.5], we deduce that $\partial_u I(\lambda, u)$ is a Fredholm operator of index 0.
Next, following [Lemma 2.18 in \cite{MR3311916}], we can demonstrate that the bounded closed sets of solutions to \eqref{ql} are compact in $[0,\infty) \times X$.

Let us define
$$
\Lambda=\sup \left\{\lambda \in [0,\infty): \eqref{ql} \text { has a solution}\right\}.
$$
We claim that $\Lambda<\infty$. Recalling assumption $(f2)$, we know that $h(t)\geq \theta t,~ \theta>0, ~\forall~ t\geq 0.$ Now multiplying \eqref{ql} by  $\varphi_1$, we get
$$
\lambda_1(\mathcal{D}) \int_{\Omega} u \varphi_1 \mathrm{dx}=\lambda \int_{\Omega} h(u) \varphi_1 \mathrm{dx} \geq \theta \lambda \int_{\Omega} u \varphi_1 \mathrm{dx}.
$$
Therefore, $\lambda \leq \frac{\lambda_1(\mathcal{D})}{\theta}$ and hence the claim.
Without loss of generality, we may assume that $h^{\prime}(0)=1$ when $f^{\prime}(0) = (1-\theta)$. Then, we can see  $T=\partial_u I\left(\lambda_1(\mathcal{D}), 0\right)=I d-\lambda_1(\mathcal{D}) K$ satisfies
$\operatorname{ker} T=\operatorname{span}\left\{\varphi_1\right\}$ and $\operatorname{dim} \operatorname{ker} T=\operatorname{codimR}(T)$,
where $\mathrm{R}(T)$ denotes the range space of the operator $T$, using [\cite{MR1956130}, Theorem 2.7.5]. Also, we can see that the transversality condition is satisfied, i.e. $\partial_{\lambda\,u}^2 I\left(\lambda_1(\mathcal{D}), 0\right) \psi \frac{\psi}{\lambda_1(\mathcal{D})} \notin R(T)$. So, with the help of well-known results of 
Crandall-Rabinowitz on the bifurcation from simple eigenvalue, see [\cite{MR1956130}, Theorem 8.3.1] in the analytic case, we obtain the existence of unique local and nontrivial analytic branch $\Gamma_0^{+}$ which emanates from $\left(\frac{\lambda_1(\mathcal{D})}{a}, 0\right)$. Hence, we can now apply the global bifurcation result[\cite{MR1956130}, Theorem 9.1.1] to obtain the existence of the branch $\Gamma_0$ extending the local branch $\Gamma_0^{+}$ where
$$
\Gamma_0=\left\{(\lambda(t), u(t)), t \in[0, \infty),(\lambda(0), u(0))=\left(\frac{\lambda_1(\mathcal{D})}{a}, 0\right)\right\}
$$
and
$$
(\lambda, u):[0, \infty) \longrightarrow [0,\infty) \times X \text { is continuous. }
$$
We show now that $\Gamma_0$ is unbounded. According to the global bifurcation theorem, if $\Gamma_0$ is bounded then either it converges to a boundary point, say $\left(0, u_0\right) \in [0,\infty) \times X$ or $\Gamma_0$ is a closed loop. This latter case is impossible since the bifurcating branch from $\left(\frac{\lambda_1(\mathcal{D})}{a}, 0\right)$ is unique. So only the first case can happen. Hence, there exists a sequence in $\Gamma$ i.e. $\left(\lambda_n, u_n\right)=\left(\lambda\left(t_n\right), u\left(t_n\right)\right) \in [0,\infty) \times X$ such that
$$
\lambda_n \rightarrow 0 \text { and } u_n \rightarrow u_0 \text { in } X \text { as } t_n \rightarrow \infty.
$$
We know $u_n\geq 0$ in $U$, due to Lemma \ref{weak-max}.
Let $n$ be large enough such that $\lambda_n<\frac{\lambda_1(\mathcal{D})}{c}$ and $s_0$ be such that $\left\|u_n\right\|_{L^{\infty}(\Omega)} \leq s_0$. By the assumptions $h(0)=0$ and $(f1)$,
$
\exists~ c>0$ such that $ \text{for all } |s|\leq s_0$, it holds $h(s) \leq cs.
$ 
Testing  \eqref{ql} by $\varphi_1$, we find that for large $n$
$$\lambda_1(\mathcal{D}) \int_{\Omega} u_n \varphi_1 \mathrm{dx}=\lambda_n \int_{\Omega} h\left(u_n\right) \varphi_1 \mathrm{dx} \leq c \lambda_n \int_{\Omega} u_n \varphi_1 \mathrm{dx}$$
 which contradicts $\lambda_n \rightarrow 0 \text {. Hence the branch } \Gamma_0 \text { is unbounded}.$
\end{proof}

\subsection{Bifurcation from infinity}
\begin{Definition}
    Suppose there exists $(\lambda_k,u_k)\in \Gamma$ 
such that $\lambda_k\to \lambda_{\infty}$ and $\|u_k\|\to \infty$, as $k\to\infty$ then $\lambda_{\infty}$  is said to be bifurcation from infinity for $I(\lambda,u)=0.$
\end{Definition}

 To identify bifurcation points from infinity, we can utilize a change of variable given by
$$
v=\frac{u}{\|u\|^2}, \text { for } u \in X \text { with } u \neq 0 \text {. }
$$
This easily implies
$
v=\frac{u}{\|u\|^2} ~\Leftrightarrow ~ u=\frac{v}{\|v\|^2}.
$
With the help of the new variable $v$, we can see that
$$
I(\lambda, u)=0 \Leftrightarrow u=\lambda K(h(u)) \Leftrightarrow \frac{v}{\|v\|^2}=\lambda K\left(h\left(\frac{v}{\|v\|^2}\right)\right) \Leftrightarrow v=\lambda\|v\|^2 K\left(h\left(\frac{v}{\|v\|^2}\right)\right).
$$
Thus, we consider the mapping $\tilde{\Psi}:[0,+\infty)  \times X \rightarrow X$ defined by 
$$
\tilde{\Psi}(\lambda, v)=\left\{\begin{array}{l}
v-\lambda\|v\|^2 K\left[h\left(\frac{v}{\|v\|^{2}} \right)\right], \text { if } v \neq 0, \\
0 ~~~~~~~~~\quad\qquad\qquad\qquad,\text { if } v=0.
\end{array}\right.
$$
Clearly, $\tilde{\Psi}$ is continuous at $v=0$ and $\tilde{\Psi}$ is a compact perturbation of identity. Furthermore, setting
$$ \Gamma_1=\{(\lambda,v): v\neq 0, \tilde{\Psi}(\lambda,v)=0\},$$ one can verify that there holds the following
\begin{equation}\label{511}
    (\lambda,u)\in\Gamma \iff (\lambda,v)\in \Gamma_1.
\end{equation}
Additionally, $\|u_k\| \to \infty \iff \|v_k\|= \|u_k\|^{-1}\to 0.$
Thus \eqref{511} implies the following remark.

\begin{remark}(See Lemma 4.13 in \cite{Ambrosetti})
    $\lambda_{\infty}$ is a bifurcation point from infinity for $I(\lambda, u)=0 \iff \lambda_{\infty}$ is a bifurcation point from trivial solutions for $\tilde{\Psi}(\lambda, v)=0$. In such a case we will say that $\Gamma$ bifurcates from $(\lambda_{\infty},\infty).$
\end{remark}

We next begin our investigation by proving that $\lambda_{\infty}$ is a bifurcation point from infinity for the positive solutions. 
\begin{Lemma}\label{l57}
\begin{enumerate}[(a)]
Suppose $\lambda<\lambda_{\infty}$. Then
\item $\exists~ R>0$,  $\forall$ $u \in X$ and $\forall~t \in[0,1] \text{such that}~ I_{t \lambda}(u) \neq 0$ whenever $\|u\| \geq R.$
\item $\lambda_{\infty}$ is the only possible bifurcation from infinity.
    \item Suppose $\lambda<\lambda_{\infty}$, then ind $(\tilde{\Psi}_\lambda, 0)=1$.
 \end{enumerate}
     \end{Lemma}
\begin{proof}
(a) Firstly, we claim that there is $R>0$ such that $\forall$ $u \in X$ and all $t \in[0,1], I_{t \lambda}(u) \neq 0$ whenever $\|u\| \geq R$. Supposing  the contrary,  there are sequences $\left\{u_k\right\} \subset X$ with $\left\|u_k\right\| \rightarrow+\infty$ and $\left\{t_k\right\} \subset[0,1]$, such that $I_{t_k \lambda}\left(u_k\right)=0$ $\forall~ k\in\mathbb{N}$. Then,
$
u_k=r_k K\left[h\left(u_k\right)\right], ~\forall~ k \in \mathbb{N},
$
where $r_k=t_k \lambda$.
By Lemma \ref{weak-max} and $(f1)$, $u_k\geq 0$ in $U$.
We set for large $k$
$$
w_k=\frac{u_k}{\left\|u_k\right\|}, ~\text{where} \left\|u_k\right\| \neq 0,
$$
 and find that it satisfies
\begin{equation}
w_k=r_k K\left[\frac{h\left(u_k\right)}{\left\|u_k\right\|}\right].
\end{equation}
We know that, the sequence $\left\{\frac{h\left(u_k\right)}{\left\|u_k\right\|}\right\} \subset X$ is bounded and by the compactness of operator $K$, 
we may suppose that $w_k \rightarrow w$ in $X$. Then, clearly $\|w\|=1$ and $w \geq 0$. Supposing that as $k\to \infty$
$$
t_{k} \rightarrow t_0 \quad \text { so that } \quad t_{k}\lambda=r_{k} \rightarrow t_0 \lambda =\bar{\lambda}\text{(say)}.
$$ 
Now, we are combining the above facts to obtain the following
\begin{equation}\label{513}
w_k=r_k K\left[\frac{\theta u_k+f(u_k)}{\left\|u_k\right\|}\right]
=r_k \theta K\left(w_k\right)+r_k K\left[\frac{f\left(u_k\right)}{\left\|u_k\right\|}\right], \forall~k \in \mathbb{N}.
    \end{equation}
Due to the boundedness of $f$, see $(f2)$, passing limit as $k \rightarrow \infty$ in \eqref{513}, we get
$
w=\bar{\lambda} \theta K(w)  \text { in } X.
$
As derived in the previous section, $w>0$ in $U$. So, $w>0$ is  eigenfunction associated with the eigenvalue $\bar{\lambda}\theta$ of operator $L$. But the only eigenvalue with a positive eigenfunction is $\lambda_1(\mathcal{D}).$  Hence, eigenfunctions $w=\varphi_1$ and $\lambda\geq \lambda_{\infty}$ which is a contradiction, since $\lambda<\lambda_{\infty}$. 

Part $(b)$ follows immediately from part $(a)$. Thus, $\lambda_{\infty}$ is the only possible bifurcation from infinity. 
Regarding part $(c)$, by using statement $(a)$ then
 for all $u \in X$ with $\|u\| \geq R$, $\forall~ t \in[0,1]$
$$I_{t\lambda}(u)\neq 0, \quad \text { for }\|u\| \geq R~~ \text{i.e.}~ u-t \lambda K[h(u)] \neq 0.$$
This implies for $v=\frac{u}{\|u\|^2}$  with $u\in X\setminus \{0\}$
$$\tilde{\Psi}_{t\lambda}(v)\neq 0,~ \forall~0< \|v\|\leq R^{-1},$$
since
$$
v-t \lambda\|v\|^2 K\left[h\left( \frac{v}{\|v\|^{2}}\right)\right] \neq 0
$$
for all $v \in X$ with $0<\|v\| \leq R^{-1}$ and $t \in[0,1]$. We consider the homotopy $H(t,u)= \tilde{\Psi}_{t\lambda}(u)$. Using  the homotopy invariance of the LeraySchauder degree, $\forall ~\varepsilon \in\left(0, R^{-1}\right]$, we get
$$
\operatorname{deg}\left(H(1,.), B_{\varepsilon}(0), 0\right)=\operatorname{deg}\left(H(0,.), B_{\varepsilon}(0), 0\right)=1,~ \text{namely}
$$
$$
\operatorname{deg}\left(\tilde{\Psi}_\lambda, B_{\varepsilon}(0), 0\right)=\operatorname{deg}\left(I, B_{\varepsilon}(0), 0\right)=1, ~\text{where}~ I= \text{identity operator}.
$$
Hence, by the definition of index(see Remark \eqref{rem5.1}), we obtain $\operatorname{ind}\left(\tilde{\Psi}_\lambda, 0\right)=1.$
\end{proof}
\begin{Lemma}\label{l58}
\begin{enumerate}[(a)]
Suppose $\lambda>\lambda_{\infty}$. Then 
    \item $\exists ~ R>0$ such that $I_\lambda(u) \neq b \varphi_1$, $\forall~b\geq 0$ whenever $\|u\|\geq R$.
    \item ind $\left(\tilde{\Psi}_\lambda, 0\right)=0$ for all $\lambda>\lambda_{\infty}.$
\end{enumerate}   
\end{Lemma}
\begin{proof}
(a) We shall initially claim that $I_\lambda(u) \neq b \varphi_1$, where $\varphi_1$ is an eigenfunction corresponding to eigenvalue  $\lambda_1(\mathcal{D})$, for any $b\geq 0$ and $u \in X \backslash\{0\}$.
Supposing the contrary, if $u \in X$ with $\|u\| \neq 0$  satisfies $I_{\lambda}(u)=b\varphi_1$ i.e. $u=\lambda K[h(u)]+b\varphi_1$, $b \geq 0$ then Lemma \ref{weak-max} says that $u\geq 0$ in $U$. Using {   Lemma \ref{Ns} and Lemma \ref{strmx}}, it follows that $u>0$ in $U$. Hence,
$$
u=\lambda  K[h(u)]+b \lambda_{1}(D) K(\varphi_1).
$$
Using \eqref{eqt}, we get that
$u$ satisfies
    $$Lu =  \lambda h(u)+b \lambda_{1}(D)  \varphi_1.$$
Using the same kind of reasoning as in the proof of 
 Lemma \ref{lem5.4}, we obtain
$$
\left(\lambda_{1}(D)-\lambda \theta\right) \int_{\Omega}  u \varphi_1 d x=\lambda \int_{\Omega}  f(u) \varphi_1 d x+b \lambda_{1}(D) \int_{\Omega}  \varphi_1^2 d x~~~~ \forall~ \varphi_1 \in \mathcal{X}^s_{D}(U).
$$
Since $f(t)>0$ for all $t>0$, it follows that
$$
\left(\lambda_{1}(D)-\lambda \theta\right) \int_{\Omega}  u \varphi_1 d x>0, ~\text{since}~ u,\varphi_1>0,
$$
then $\lambda_{1}(D)-\lambda\theta>0$ which implies that $\lambda_{1}(D)>\lambda \theta$ or $\lambda_{\infty}>\lambda$, while in our assumption $\lambda>\lambda_{\infty}$, that is a contradiction.

Now, we shall prove part $(b)$ by using part $(a)$. 
 If $\forall~t \in[0,1]$, $b=t\|u\|^2$ and all $u \in X$ with $\|u\| \geq 1$ (choosing $R=1$), such that $I_{\lambda}(u)\neq t\|u\|^2\varphi_1,$ then
\begin{equation}\label{514}
 u-\lambda K[h(u)]-t\|u\|^2 \varphi_1 \neq 0 .
   \end{equation}
   This implies \begin{equation}\label{525}
   \tilde{\Psi}_{\lambda}(v)\neq t\varphi_1, \forall~ 0<\|v\|\leq 1, ~\forall~ t \in[0,1],    
   \end{equation}
since, using the change of variable such that $u=\frac{v}{\|v\|^2}$ in \eqref{514}, we have
$$
v-\lambda\|v\|^2 K\left[h\left(\frac{v}{\|v\|^{2}} \right)\right]-t \varphi_1 \neq 0.
$$
Using the homotopy $H(t,v)=\tilde{\Psi}_{\lambda}(v)-t\varphi_1$, \text{on the ball}~$B_{\varepsilon}(0)$, we get
$$
\operatorname{deg}\left(H(1,.), B_{\varepsilon}(0), 0\right)=\operatorname{deg}\left(H(0,.), B_{\varepsilon}(0), 0\right)=0,~ \text{namely}
$$
$$
\operatorname{deg}\left(\tilde{\Psi}_\lambda, B_{\varepsilon}(0), 0\right)=\operatorname{deg}\left(\tilde{\Psi}_\lambda-\varphi_1, B_{\varepsilon}(0), 0\right)=0,
$$
for all $\varepsilon \in(0,1]$.
The latter degree is zero because \eqref{525}, with $t=1$, implies that $\tilde{\Psi}_{\lambda}(v)=\varphi_1$ has no solution on the ball $B_{\varepsilon}.$ Hence, by the definition of index ( see Remark \eqref{rem5.1}), we obtain $\operatorname{ind}\left(\tilde{\Psi}_\lambda, 0\right)=0$. So, our proof is done.
    \end{proof}

\textbf{Proof of Theorem \ref{thm52}:}
Firstly, arguing as in the proof of Theorem \ref{thm210} and we will show $\lambda_{\infty}$ is a unique bifurcation point from the trivial solutions for the equation $\tilde{\Psi}_\lambda(v)=0$ which ensure the existence of $\Gamma_{\infty}$. To begin, we need to demonstrate that $\lambda_{\infty}$  is a unique bifurcation point from the trivial solutions for the equation $\tilde{\Psi}_{\lambda}(v)=0$.
Otherwise, using Lemma \ref{l57} (a), we have
$$
\tilde{\Psi}_\lambda(v) \neq 0, \forall~ \lambda \in\left[\lambda_0-\frac{1}{R}, \lambda_0+\frac{1}{R}\right] \text { and } 0<\|u\|\leq \frac{1}{R} , \text{for some}~R>0.
$$
 Thus, there exist $\mu_1$ and $\mu_2$ such that
$$
\lambda_0-\frac{1}{R}<\mu_1<\lambda_0<\mu_2<\lambda_0+\frac{1}{R}
$$
and
$$
\operatorname{deg}\left(\tilde{\Psi}_{\mu_1}, B_{\frac{1}{R}}(0), 0\right)=\operatorname{deg}\left(\tilde{\Psi}_{\mu_2}, B_{\frac{1}{R}}(0), 0\right),
$$
therefore,
$$
\operatorname{ind}\left(\tilde{\Psi}_{\mu_1}, 0\right)=\operatorname{ind}\left(\tilde{\Psi}_{\mu_2}, 0\right)
$$
 that contradicts the Lemmas \ref{l57} and \ref{l58}. Hence, $\lambda_{\infty}$ is a bifurcation point of $\tilde{\Psi}_{\lambda}(v)=0$. Furthermore, the reasoning employed in Lemma \ref{l57} guarantee that $\lambda_{\infty}$ is the only bifurcation point from the trivial solutions for $\tilde{\Psi}_{\lambda}(v)=0$,
  and from $(\lambda_{\infty},0)$ emanates an unbounded continuum of solutions. Using the change of  variables $u=\frac{v}{\|v\|^{2}}$,  $v \neq 0$, the existence of $\Gamma_{\infty}$ is proved.


\begin{remark}
We remark that the theory developed above is applicable to the following functions as $h$ and the bifurcation curve shall look as below-
\begin{enumerate}
    \item  Suppose $h(s) = s + s^2e^{-s}$ such that $h'(0) = 1$ and $h(s)>s$ for $s>0$.
$$
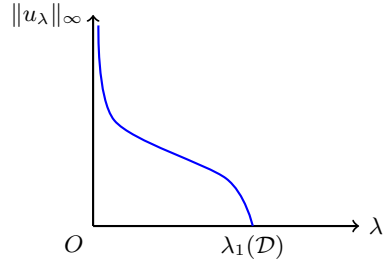
\begin{tikzpicture}[scale=0.7]
\draw[->] [thick] (0,0) node[below left]{$O$} -- (5,0) node[right]{$\lambda$};
\draw[->] [thick] (0,0)--(0,4) node[left]{$\|u_\lambda\|_\infty$};
\node (P) at (3,0) [below]{$\lambda_1(\mathcal{D})$};
\draw [thick, color=blue] plot [smooth] coordinates {(0.1,3.8) (0.4,2) (2.5,0.9) (3,0) };
\end{tikzpicture}
 $$
 \item Suppose $h(s) = s - s^p$ ($1<p<\infty$, i.e. logistic type nonlinearity). Then $h$ is a concave function such that $h'(0) = 1,$ and $h(s)<s$ for $s>0$ small.
$$
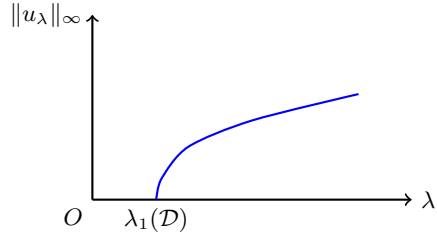
\begin{tikzpicture}[scale=0.7]
\draw[->] [thick] (0,0) node[below left]{$O$} -- (6,0) node[right]{$\lambda$};
\draw[->] [thick] (0,0)--(0,3.5) node[left]{$\|u_\lambda\|_\infty$};
\node (P) at (1.2,0) [below]{$\lambda_1(\mathcal{D})$};
\draw [thick, color=blue] plot [smooth] coordinates {(1.2,0) (1.3,0.4) (1.8,1)  (3,1.5) (5,2) };
\end{tikzpicture}
 $$
\end{enumerate}
The curves above represent $\|u_{\lambda}\|_{\infty}$ with respect to $\lambda$, whenever $(\lambda,u) $ is solution of \eqref{2} with above $h$.
     \end{remark}


   
\section{Appendix}\label{app}

The primary objective of this section is to establish regularity results, which are based on the $W^{2, p}$ theory for $\mathcal{L}$ as formulated by  Lions et al. in \cite{Bensoussan}. The following result is important to finish the proof of Lemma \ref{cpto}.

\begin{theorem}\label{thm5.2}
Suppose $w \in L^{\infty}(U)$ and  $\partial U$ is of class $C^{1,1}$.  Also let  $u \in \mathcal{X}^{1,2}_{\mathcal{D}}(U)$ be weak solution of \eqref{w} then
$ u \in C^{1, \beta}(\bar\Omega)$, for some $\beta \in(0,1)$.
    \end{theorem}
\begin{proof}
It is easy to see that $w\in L^{\infty}(U)$ implies $w \in L^p(U)$, $\forall$~ $p \geq 2$, since $U$ is bounded. We utilize  [Theorem 3.1.22 in \cite{garroni2002second}] to achieve $W^{2, p}$ regularity for some $p > 1$ i.e. $\|u\|_{W^{2,p}(\Omega)}\leq C_1\|w\|_{L^p(\Omega)}$ and combining this with  compact embedding of $W^{2, p}(\Omega)$ in $C^{1,\beta}(\bar \Omega)$ for $p > n$, we get $ u \in C^{1, \beta}(\bar\Omega)$. 
\end{proof}

\begin{Lemma}\label{Ns}
 Let $u\in  \mathcal{X}^{1,2}_{\mathcal{D}}(U) $.   Suppose $\mathcal{N}_su=0$ in $ \mathcal{N}$ then $u\in C^{0,\beta}(\mathbb{R}^n)$ for some $\beta\in (0,1)$.
\end{Lemma}
\begin{proof}
    From $\mathcal{N}_s u=0$ in $\mathcal{N}$,  we have
$$ u(x)\int_{\Omega} \frac{dy}{|x-y|^{n+2s}}= \int_{\Omega} \frac{u(y) dy}{|x-y|^{n+2s}}.
$$
 In the above equality, we observe that both parametrized integrals with respect to $x$ are differentiable in $\mathcal{N}$, since for any $x\in\mathcal{N}$, $\text{dist} (x,\Omega)>0$. This implies $u$ is differentiable in  $\mathcal{N}.$ Next, from this, Theorem \ref{thm5.2}  and $u=0$ in $\mathcal{D}$, we obtain that $u \in C^{0,\beta}(\mathbb{R}^n).$
\end{proof}

We present a few more interesting properties in the same direction.
First, we recall the following continuity result for functions satisfying the nonlocal Neumann condition, see [Proposition 5.2 in \cite{dipierro2017nonlocal}].
\begin{Proposition}\label{pp5.2}
     Let $\Omega\subset \mathbb{R}^n$ be a domain with $C^1$ boundary. Let $u$ be continuous
in $\Omega$, with $\mathcal{N}_s(u)= 0 ~\text{in}~ \mathbb{R}^n\setminus \bar \Omega$. Then $u$ is continuous in the whole of $\mathbb{R}^n.$
\end{Proposition}

\begin{Corollary}\label{cor6.2}
 Let $\Omega \subset \mathbb{R}^n$ be a domain with $C^1$ boundary and  $w \in C\left(\mathbb{R}^n\right)\cap \mathcal{X}^{1,2}_{\mathcal{D}}(U)$. Suppose
$$
u(x)=\left\{\begin{array}{cl}
w(x) & \text { if } x \in \bar{\Omega},\vspace{0.5em} \\
\frac{\int_{\Omega} \frac{w(y)}{|x-y|^{n+2 s}} d y}{\int_{\Omega} \frac{d y}{|x-y|^{n+2 s}}} & \text { if } x \in \mathcal{N}\subset \mathbb{R}^n \backslash \bar{\Omega}.
\end{array}\right.
$$
Then $u \in C\left(\mathbb{R}^n\right)$.
\end{Corollary}
\begin{proof}
 By assumption, we have $u=w$ in $\bar{\Omega}$ and $\mathcal{N}_s u=0$ in $\mathcal{N}\subset\mathbb{R}^n \backslash \bar{\Omega}$. Thanks to Proposition \ref{pp5.2} and recalling $u=0$ in $\mathcal{D}$, since $u\in \mathcal{X}^{1,2}_{\mathcal{D}}(U)$, we easily deduce that  $u \in C\left(\mathbb{R}^n\right)$.
    \end{proof}
    Our focus is on the boundary behaviour of the nonlocal Neumann function $\tilde{\mathcal{N}_s} v$ defined as
    $$\tilde{\mathcal{N}_s} v(x)=\frac{\mathcal{N}_s v(x)}{\int_{\Omega} \frac{d y}{|x-y|^{n+2 s}}}.$$
 
\begin{Lemma}
   Let $\Omega \subset \mathbb{R}^n$ be a domain with $C^1$ boundary and $v \in C\left(\mathbb{R}^n\right)\cap\mathcal{X}^{1,2}_{\mathcal{D}}(U)$. Then
$$
\lim _{\substack{x \rightarrow \partial \Omega \\ x \in \mathcal{N}\subset\mathbb{R}^n\setminus \bar{\Omega}}} \tilde{\mathcal{N}}_s v(x)=0, ~\text{for every}~ s\in (0,1).
$$
\end{Lemma}

\begin{proof}
 We consider a sequence  $x_k \in \mathcal{N}$ such that $x_k \rightarrow x \in \partial \Omega$ as $k \rightarrow+\infty$.
Using the Corollary \ref{cor6.2} with the notation $w=v$, there exists $u \in C\left(\mathbb{R}^n\right)$ such that $u=v$ in $\bar{\Omega}$ and $\mathcal{N}_s u=0$ in $\mathcal{N}\subset\mathbb{R}^n \backslash \bar{\Omega}$. By the continuity of $u$ and $v$ we have that
\begin{equation}\label{e601}
\lim _{k \rightarrow+\infty} v\left(x_k\right)-u\left(x_k\right)=v\left(x\right)-u(x)=0 \text {. }
    \end{equation}

Moreover,
$$
\begin{aligned}
\tilde{\mathcal{N}}_s v\left(x_k\right) & =\tilde{\mathcal{N}}_s v\left(x_k\right)-\tilde{\mathcal{N}}_s u\left(x_k\right)=\frac{\int_{\Omega} \frac{v\left(x_k\right)-v(y)}{\left|x_k-y\right|^{n+2 s}} d y-\int_{\Omega} \frac{u\left(x_k\right)-u(y)}{\left|x_k-y\right|^{n+2 s}} d y}{\int_{\Omega} \frac{d y}{\left|x_k-y\right|^{n+2 s}}} \\
& =\frac{\int_{\Omega} \frac{v\left(x_k\right)-u\left(x_k\right)}{\left|x_k-y\right|^{n+2 s}} d y}{\int_{\Omega} \frac{d y}{\left|x_k-y\right|^{n+2 s}}}=\int_{\Omega}v\left(x_k\right)-u\left(x_k\right)\,dy .
\end{aligned}
$$
This and \eqref{e601} imply that $\lim_{k\to \infty} \tilde{\mathcal{N}_s}v(x_k)=0$. Thus, our proof is done.
    \end{proof}

\section*{Acknowledgements} 
Tuhina Mukherjee acknowledges the financial support provided by CSIR-HRDG with sanction No. 25/0324/23/EMR-II. Jacques Giacomoni was partially funded by IFCAM (Indo-French Centre for Applied Mathematics) IRL CNRS 3494.	Lovelesh Sharma received assistance from the UGC Grant with reference no. 191620169606 funded by the Government
of India.	

\Addresses

\end{document}